\documentclass[final,leqno,letterpaper]{etna}

\setbibdata{0}{0}{0}{2021} %First page/Last page/volume/year

\hypersetup{%
    pdftitle={Decay bounds for Bernstein functions of Hermitian matrices with applications to the fractional graph Laplacian},
    pdfauthor={Marcel Schweitzer},
    pdfkeywords={matrix functions, Bernstein functions, off-diagonal decay, graph Laplacian, fractional powers, nonlocal dynamics}
    }

\usepackage[nosumlimits]{amsmath}
\usepackage{amsfonts,amssymb}
\usepackage{graphicx}
\graphicspath{{./Figures/}}
\usepackage{color}
\usepackage{ulem}
\usepackage{multicol}
\usepackage{algorithm2e}
\usepackage[showonlyrefs]{mathtools}
\usepackage{cite} % order citations [1-3,5]
\usepackage{placeins}
\usepackage{pifont}
\usepackage[dvipsnames]{xcolor}
\usepackage{tikz}
\usepackage{pgfplots}
\usetikzlibrary{shapes}
\DeclareMathAlphabet{\mathbf}{OT1}{cmr}{bx}{it}

\newcommand{\vv}{{\mathbf v}}

\newcommand{\lmin}{\lambda_{\min}}
\newcommand{\lmax}{\lambda_{\max}}

\newcommand{\dmu}{\d\mu(t)}

\renewcommand{\d}{\,\mathrm{d}}
\newcommand{\N}{\mathbb{N}}
\newcommand{\R}{\mathbb{R}}

\newcommand{\Rnn}{\mathbb{R}^{n \times n}}

\newcommand{\C}{\mathbb{C}}

\newcommand{\Cnn}{\mathbb{C}^{n \times n}}

\DeclareMathOperator{\erf}{erf}
\DeclareMathOperator{\erfc}{erfc}

\DeclareMathOperator{\tridiag}{tridiag}

\newtheorem{examplesimple}[theorem]{Example}
\let\oldexamplesimple\examplesimple
\renewcommand{\examplesimple}{\oldexamplesimple\normalfont}
\newenvironment{example}{\begin{examplesimple}}{\hfill$\diamond$\end{examplesimple}}
\newtheorem{remarksimple}[theorem]{Remark}
\let\oldremarksimple\remarksimple
\renewcommand{\remarksimple}{\oldremarksimple\normalfont}
\newenvironment{remark}{\begin{remarksimple}}{\hfill$\diamond$\end{remarksimple}}
\newtheorem{experiment}[theorem]{Experiment}
\let\oldexperiment\experiment
\renewcommand{\experiment}{\oldexperiment\normalfont}

\pgfplotsset{compat=1.12}

\title{Decay bounds for Bernstein functions of Hermitian matrices with applications to the fractional graph Laplacian\thanks{%
Received... Accepted... Published online on... Recommended by....}}
\author{Marcel Schweitzer\footnotemark[2]}
\date{\today}

\shorttitle{DECAY IN BERNSTEIN FUNCTIONS OF HERMITIAN MATRICES} 
\shortauthor{M.~SCHWEITZER}
\begin{document}

\maketitle 

\renewcommand{\thefootnote}{\fnsymbol{footnote}}

\footnotetext[2]{School of Mathematics and Natural Sciences, Bergische Universit\"at Wuppertal, 42097 Wuppertal, Germany, \texttt{marcel@uni-wuppertal.de}}

\begin{abstract}
For many functions of matrices $f(A)$, it is known that their entries exhibit a rapid---often exponential or even superexponential---decay away from the sparsity pattern of the matrix $A$. In this paper we specifically focus on the class of Bernstein functions, which contains the fractional powers $A^\alpha$, $\alpha \in (0,1)$ as an important special case, and derive new decay bounds by exploiting known results for the matrix exponential in conjunction with the L\'evy--Khintchine integral representation. As a particular special case, we find a result concerning the power law decay of the strength of connection in nonlocal network dynamics described by the fractional graph Laplacian, which improves upon known results from the literature by doubling the exponent in the power law.
\end{abstract}



\section{Introduction}\label{sec:intro}
In this paper, we investigate decay in the entries of matrix functions $f(A)$, where $A\in\Cnn$ is a Hermitian matrix and $f$ is a \emph{Bernstein function}, i.e., a nonnegative function $f: (0,\infty) \rightarrow \R$ which is infinitely many times continuously differentiable and satisfies
\begin{equation}\label{eq:bernstein_condition_intro}
(-1)^{n-1}f^{(n)}(z) \geq 0 \text{ for all } n \in \N \text{ and } z \in (0, \infty).
\end{equation}
As the condition~\eqref{eq:bernstein_condition_intro} means that $f^\prime$ is a completely monotonic function, the class of Bernstein functions is intimately related to the completely monotonic function classes of Laplace--Stieltjes and Cauchy--Stieltjes transforms. While the latter classes have received considerable interest in the analysis of matrix functions in recent years, see, e.g.,~\cite{BeckermannReichel2009,FrommerGuettelSchweitzer2014b,GuettelKnizhnerman2013,GuettelSchweitzer2021, MasseiRobol2020} and the references therein, the class of Bernstein functions has not been investigated as thoroughly, although it also frequently occurs in applications: Our present study is motivated by the fact that fractional powers $L_G^\alpha$, $\alpha \in (0,1)$, where $L_G$ is the Laplacian of an undirected graph $G$, have recently emerged as a useful tool in modeling non-local diffusion processes on graphs and in the efficient exploration of large networks; see, e.g.,~\cite{BenziBertacciniDurastanteSimunec2019,BianchiDonatelliDurastanteMazza2021,RiascosMateos2014,RiascosMateos2019}. Clearly, $z^\alpha, \alpha \in (0,1)$ is nonnegative on $(0,\infty)$ and fulfills the condition~\eqref{eq:bernstein_condition_intro}, so that it is a Bernstein function, whereas it is neither a Laplace--Stieltjes nor a Cauchy--Stieltjes function.

Off-diagonal decay in matrix functions is a topic that has been intensively studied in the past, in particular for the case of the matrix inverse $f(A) = A^{-1}$, see, e.g.,~\cite{DemkoMossSmith1984, EijkhoutPolman1988, FordSavostyanovZamarashkin2014, FrommerSchimmelSchweitzer2018, FrommerSchimmelSchweitzer2018b} and for entire functions like the exponential~\cite{BenziGolub1999,LopezPugliese2005,Iserles2000,BenziSimoncini2015,PozzaSimoncini2019}. Quite recently, some of these bounds were also extended to Cauchy--Stieltjes and Laplace--Stieltjes functions of matrices, see, e.g,~\cite{BenziSimoncini2015,FrommerSchimmelSchweitzer2018} and also the recent survey~\cite{Benzi2016}. A priori knowledge of the decay in matrix functions has many different applications, e.g., the efficient construction of sparse approximations~\cite{GiscardLuiThwaiteJaksch2015,Schimmel2020}, the design of linearly scaling algorithms for certain linear algebra problems~\cite{BenziRazouk2007,BowlerMiyazaki2012} and the analysis of probing methods for trace estimation~\cite{FrommerSchimmelSchweitzer2021}. In the context of fractional powers of the graph Laplacian $L_G$, decay estimates can give insight into transition probabilities of non-local random walks on $G$; see~\cite{BenziBertacciniDurastanteSimunec2019, BianchiDonatelliDurastanteMazza2021}. 

The remainder of this paper is organized as follows. In Section~\ref{sec:basics}, we recall some basic facts about Bernstein functions and several special functions that appear in the bounds that we derive in later sections and on functions of matrices in general. Our main results on the decay in Bernstein functions of matrices are presented in Section~\ref{sec:decay}, where we consider both the case of positive semidefinite as well as the case of positive definite matrices $A$. Section~\ref{sec:fractional_laplacian} deals with the special case of the fractional graph Laplacian, where we compare our new decay estimates to previously found estimates from the literature. Concluding remarks are given in Section~\ref{sec:conclusions}.

\section{Basics}\label{sec:basics} 
In this section, we introduce the basic concepts and notations needed for the derivations in later sections of the paper.

\subsection{Bernstein functions}\label{subsec:bernstein_functions}
A \emph{Bernstein function} is a nonnegative function $f: (0,\infty) \rightarrow \R$ which is infinitely many times continuously differentiable and satisfies~\eqref{eq:bernstein_condition_intro}. As already mentioned in Section~\ref{sec:intro}, this implies that $f^\prime$ is a completely monotonic function. Bernstein functions can thus be characterized as nonnegative primitives of completely monotonic functions. An important result on Bernstein functions is that they exhibit the \emph{L\'evy--Khintchine integral representation}
\begin{equation}\label{eq:bernstein_function}
f(z) = a + bz + \int_0^\infty (1-e^{-tz}) \dmu
\end{equation}
where $a,b \geq 0$ and $\mu$ is a positive measure (the \emph{L\'evy measure}) on $(0,\infty)$ such that
\begin{equation*}
\int_0^\infty \min\{t,1\} \dmu < \infty,
\end{equation*}
see, e.g.,~\cite{Berg2007,SchillingSongVondracek2012}. Also note that any Bernstein function admits a continuous extension to the origin (which we also denote by $f$ for convenience) for which the L\'evy-Khintchine representation remains valid, see, e.g.~\cite[Proof of Proposition~3.6]{SchillingSongVondracek2012}.

Important examples of Bernstein functions are, e.g.,
\begin{itemize}
\item $f(z) = z^\alpha$, $\alpha \in (0,1)$,
\item $f(z) = 1-e^{-tz}$, $t \geq 0$ and
\item $f(z) = \log(1+z)$.
\end{itemize}
As the composition of two Bernstein functions is again a Bernstein function, we also have that, e.g., $1-e^{-tz^\alpha}$, $t \geq 0$, $\alpha \in (0,1)$ is a Bernstein function. 

Because they are of particular importance in the applications we consider in later sections, we mention that the L\'evy-Khintchine representation of the fractional powers is explicitly known and given by
\begin{equation}\label{eq:fractional_power_bernstein}
z^\alpha = \frac{\alpha}{\Gamma(1-\alpha)} \int_0^\infty (1-e^{-tz})t^{-\alpha-1} \d t, \qquad \alpha \in (0,1).
\end{equation}
where $\Gamma$ denotes the gamma function; cf.~Section~\ref{subsec:special_functions}

\subsection{The Gamma function and related special functions}\label{subsec:special_functions}
In the following, we introduce some classical special functions which appear in the derivation of our results.

The \emph{gamma function} is defined for $z \in \C$ with $\Re(z) > 0$, where $\Re(z)$ denotes the real part of $z$, via
\begin{equation*}
\Gamma(z) = \int_0^\infty t^{z-1}e^{-z} \d z
\end{equation*}
and has the property that $\Gamma(n) = (n-1)\cdot\Gamma(n-1) = (n-1)!$ for $n \in \N$. Closely related are the \emph{upper} and \emph{lower incomplete gamma function}, defined by
\begin{equation*}
\Gamma(z,s) = \int_s^\infty t^{z-1}e^{-t} \d t \quad\text{ and }\quad \gamma(z,s) = \int_0^s t^{z-1}e^{-t} \d t,
\end{equation*}
respectively. Clearly, we have $\Gamma(z) = \Gamma(z,s) + \gamma(z,s) \text{ for all } s \geq 0$ and $\Gamma(z) = \Gamma(z,0) = \lim\limits_{s \rightarrow \infty} \gamma(z,s).$ We also need the \emph{error function}
\begin{equation*}
\erf(z) = \frac{1}{\sqrt{\pi}}\int_{-z}^z e^{-t^2} \d t
\end{equation*}
and the \emph{complementary error function} $\erfc(z) = 1-\erf(z)$, which are related to the incomplete gamma functions through the identities
\begin{equation*}
\Gamma\left(z,\frac12\right) = \sqrt{\pi}\erfc(\sqrt{z}) \quad\text{ and }\quad \gamma\left(z,\frac12\right) = \sqrt{\pi}\erf(\sqrt{z}).
\end{equation*}

\subsection{Matrix functions and graphs of matrices}\label{subsec:matrix_functions}
Let $A \in \Cnn$ be a Hermitian matrix with eigendecomposition $A = V\Lambda V^H$, where $\Lambda = \diag(\lambda_1,\dots,\lambda_n)$ is the diagonal matrix of eigenvalues and $V$ contains the corresponding orthonormal eigenvectors. Then, for a scalar function $f: \C \rightarrow \C$, the matrix function $f(A)$ is given by the simple relation $f(A) = Vf(\Lambda)V^H$, where $f(\Lambda) = \diag(f(\lambda_1),\dots,f(\lambda_n))$, provided that $f(\lambda_i)$ exists for all $i = 1,\dots,n$. For general---not necessarily diagonalizable---$A$, a similar definition using the Jordan canonical form is possible (where for eigenvalues with nontrivial Jordan blocks, also the derivatives of $f$ need to be defined, up to the block size minus $1$). As we only consider Hermitian matrices in this work, we do not further pursue this topic and refer the reader to~\cite[Chapter 1.2]{Higham2008} for details.

It directly follows from the definition of matrix functions given above that when $f$ is a Bernstein function~\eqref{eq:bernstein_function}, we can insert $A$ in place of $z$ into the integral representation and find
\begin{equation}\label{eq:bernstein_function_A}
f(A) = aI + bA + \int_0^\infty (I-e^{-tA}) \dmu.
\end{equation}

An important concept frequently used in the formulation of decay bounds for matrix functions is that of the \emph{graph of a sparse matrix}: Given $A\in \Cnn$, the graph of $A$ is given by $G(A) = (V, E)$ where $V = \{1,\dots,n\}$ and E = $\{(i,j) : a_{ij} \neq 0, i \neq j\}$. For any two nodes $i, j$ in the graph of $A$, we denote by $d(i,j)$ the \emph{geodesic distance} of the nodes in $G(A)$, i.e., the length of the shortest path from $i$ to $j$. If there is no path from $i$ to $j$ in $G$, then we set $d(i,j) = \infty$. Clearly, when $A$ is Hermitian, $G(A)$ is undirected and thus $d(i,j) = d(j,i)$ for all $i,j$.

\section{Decay in Bernstein functions of Hermitian matrices}\label{sec:decay}

This section contains our main results on off-diagonal decay in Bernstein functions of Hermitian matrices. We first consider the case of positive semidefinite $A$ in Section~\ref{subsec:pos_semidefinite} and then discuss how the estimates can be improved when $A$ is positive definite in Section~\ref{subsec:positive_definite}. As these estimates for positive definite $A$ have the drawback that some integrals occur for which no closed form solution is available, we also derive other, more explicit, bounds for fractional powers of positive definite matrices in Section~\ref{subsec:explicit_bounds_from_cauchy_stieltjes} by exploiting a connection to Cauchy--Stieltjes functions.

\subsection{The positive semidefinite case}\label{subsec:pos_semidefinite}
Using the representation~\eqref{eq:bernstein_function} allows us to relate decay in matrix Bernstein functions to decay in the matrix exponential, a very thoroughly studied topic. In particular, our analysis in this section is based on the following theorem from~\cite{BenziSimoncini2015} on the decay in the matrix exponential, which is in turn based on the well-known convergence result of Hochbruck and Lubich~\cite[Theorem 2]{HochbruckLubich1997} for Lanczos approximations of the action of the matrix exponential. 

\begin{theorem}[Theorem 4.2 in~\cite{BenziSimoncini2015}]\label{the:benzisimoncini_exp}
Let $A \in \Cnn$ be a Hermitian positive semidefinite matrix with eigenvalues in the interval $[0, 4\rho]$ and denote by $d(i,j)$ the geodesic distance of the nodes $i$ and $j$ in the graph of $A$. Then for $i \neq j$ 
\begin{itemize}
\item[(i)] for $\rho t \geq 1$ and $\sqrt{4\rho t} \leq d(i,j) \leq 2\rho t$,
$$|[\exp(-t A)]_{ij}| \leq 10 \exp\left(-\frac{1}{5\rho t}d(i,j)^2\right),$$
\item[(ii)] for $d(i,j) \geq 2\rho t$
$$|[\exp(-t A)]_{ij}| \leq 10 \frac{\exp(-\rho t)}{\rho t}\left(\frac{e\rho t}{d(i,j)}\right)^{d(i,j)}.$$
\end{itemize}
\end{theorem}

Note that in its original form, the result of Theorem~\ref{the:benzisimoncini_exp} was stated for banded matrices, but it directly generalizes to arbitrary sparse matrices; see also~\cite[Section~5]{BenziSimoncini2015}. In~\cite{BenziSimoncini2015}, Theorem~\ref{the:benzisimoncini_exp} was used by Benzi and Simoncini to analyze the decay behavior of Laplace--Stieltjes matrix functions. Similar to Bernstein functions, Laplace--Stieltjes functions can be defined using an integral transform involving exponentials. Many of the arguments we use in the derivation of our results closely follow the techniques used in~\cite{BenziSimoncini2015}.

By exploiting the relation between Bernstein functions~\eqref{eq:bernstein_function} and the matrix exponential, we can prove the following result. 
\begin{lemma}\label{lem:bernstein_decay_integral}
Let $f$ be a Bernstein function~\eqref{eq:bernstein_function}, let $A \in \Cnn$ be positive semidefinite with spectral radius $\rho(A)$ and denote by $d(i,j)$ the geodesic distance of the nodes $i$ and $j$ in the graph of $A$. Then for all $i,j$ with $d(i,j) \geq 2$, we have
\begin{eqnarray}
|[f(A)]_{ij}| &\leq& 10\int_0^{\frac{2d(i,j)}{\rho(A)}} \frac{4\exp(-\frac{1}{4} \rho(A) t)}{\rho(A) t}\left(\frac{e \rho(A) t}{4d(i,j)}\right)^{d(i,j)}\dmu \nonumber\\
& &+ 10\int_{\frac{2d(i,j)}{\rho(A)}}^{\frac{d(i,j)^2}{\rho(A)}} \exp\left(-\frac{4d(i,j)^2}{5\rho(A) t}\right) \dmu \nonumber\\
& & + \int_{\frac{d(i,j)^2}{\rho(A)}}^\infty \left|[\exp(-tA)]_{ij}\right|\dmu.\label{eq:bernstein_bound_exp}
\end{eqnarray}
\end{lemma}
\begin{proof}
For $i \neq j$, we have
\begin{equation*}
\left|[I - \exp(-tA)]_{ij}\right| = |\left[\exp(-tA)\right]_{ij}|,
\end{equation*}
so that~\eqref{eq:bernstein_function_A} implies 
\begin{equation}\label{eq:fractional_laplacian_bounds}
|[f(A)]_{ij}| \leq |b\cdot a_{ij}| + \int_0^\infty |\left[\exp(-tA)\right]_{ij}|\dmu.
\end{equation}
We can therefore use bounds for entries of the matrix exponential in order to bound~\eqref{eq:fractional_laplacian_bounds}. Further note that for $i, j$ with $d(i,j) \geq 2$ it directly follows that $|b\cdot a_{ij}| = 0$, so that we can ignore this term. Recasting the conditions on $t$ in Theorem~\ref{the:benzisimoncini_exp} as $\frac{d(i,j)}{2\rho} \leq t \leq \frac{d(i,j)^2}{4\rho}$ or $t \leq \frac{d(i,j)}{2\rho}$, respectively, and writing $\rho = \frac14\rho(A)$, the assertion of the lemma follows.
\end{proof}

%\begin{remark}\label{rem:posdef}
%When $M$ is positive definite, one can consider the positive semidefinite matrix $\widetilde{M} = %M-\lmin I$, where $\lmin$ is the smallest eigenvalue of $M$, when applying %Lemma~\ref{lem:bernstein_decay_integral}. Noting that 
%$$\exp(-tM) = \exp(-t(\widetilde{M}+\lmin I)) = \exp(-t\lmin)\exp(t\widetilde{M}),$$
%one obtains an additional factor $\exp(-t\lmin)$ in all three integrals %in~\eqref{eq:bernstein_bound_exp}, leading to faster decay estimates. We will not pursue this %further here as we are mainly interested in the graph Laplacian which is always singular.
%\end{remark}

The integral representation~\eqref{eq:bernstein_bound_exp} does not give a clear picture of the actual decay behavior at first sight, and in general it can only be evaluated by numerical quadrature. For fractional powers $f(z) = z^\alpha, \alpha \in (0,1)$, the special case that we are most interested in, and in particular for the square root $f(z) = \sqrt{z}$, we can give analytic expressions for all occurring integrals in terms of the special functions introduced in Section~\ref{subsec:special_functions}.

\begin{theorem}\label{the:decay_z_alpha}
Let $A \in \Cnn$ be positive semidefinite with spectral radius $\rho(A)$, let $\alpha \in (0,1)$ and denote by $d(i,j)$ the distance of the nodes $i$ and $j$ in the graph of $A$. Then for all $i,j$ with $d(i,j) \geq 2$, we have
\begin{eqnarray}
|[A^\alpha]_{ij}| &\leq& \frac{\alpha}{\Gamma(1-\alpha)}\cdot\Bigg(\frac{10e^{d(i,j)}\rho(A)^\alpha}{4^\alpha d(i,j)^{d(i,j)}} \cdot \gamma\left(d(i,j)-\alpha-1,\frac{d(i,j)}{2}\right) \nonumber\\ 
& &+ 10\left(\frac{5\rho(A)}{4d(i,j)^2}\right)^\alpha \!\! \cdot \left(\Gamma\left(\alpha,\frac{4}{5}\right) - \Gamma\left(\alpha,\frac{2d(i,j)}{5}\right)\right) +\frac{\rho(A)^\alpha}{\alpha \cdot d(i,j)^{2\alpha}}\Bigg).\label{eq:bounds_z_alpha}
\end{eqnarray}
In particular, for $\alpha = \frac12$ we have
\begin{eqnarray}
|[\sqrt{A}]_{ij}| &\leq& \frac{1}{2\sqrt{\pi}}\cdot\Bigg(\frac{10e^{d(i,j)}\sqrt{\pi\rho(A)}}{2d(i,j)^{d(i,j)}}\cdot\gamma\left(d(i,j) - \frac{3}{2}, \frac{d(i,j)}{2}\right) \nonumber\\
& & + \frac{5\sqrt{5\pi\rho(A)}}{d(i,j)}\left(\erfc\left(\frac{2}{\sqrt{5}}\right)-\erfc\left(\sqrt{\frac{2d(i,j)}{5}}\right)\right) +\frac{2\sqrt{\rho(A)}}{d(i,j)}\Bigg).\label{eq:bounds_squareroot}
\end{eqnarray}
\end{theorem}
\begin{proof}
We use the representation~\eqref{eq:fractional_power_bernstein} of fractional powers and insert the result of Lemma~\ref{lem:bernstein_decay_integral}. For the third integral, note that 
$$\left|[\exp(-tA)]_{ij}\right| \leq \|\exp(-tA)\|_2 = \exp(-t\rho(A)) \leq 1$$
because $A$ is positive semidefinite. This way, we obtain
\begin{eqnarray}
|[A^\alpha]_{ij}| &\leq& \frac{\alpha}{\Gamma(1-\alpha)} \cdot \bigg( 10\int_0^{\frac{2d(i,j)}{\rho(A)}} \frac{4\exp(-\frac{1}{4} \rho(A) t)}{\rho(A) t}\left(\frac{e \rho(A) t}{4d(i,j)}\right)^{d(i,j)}\cdot t^{-\alpha-1} \d t \nonumber\\
& &\phantom{\frac{\alpha}{\Gamma(1-\alpha)} \cdot \bigg(} + 10\int_{\frac{2d(i,j)}{\rho(A)}}^{\frac{d(i,j)^2}{\rho(A)}} \exp\left(-\frac{4d(i,j)^2}{5\rho(A) t}\right)\cdot t^{-\alpha-1} \d t\nonumber\\
& &\phantom{\frac{\alpha}{\Gamma(1-\alpha)} \cdot \bigg(} + \int_{\frac{d(i,j)^2}{\rho(A)}}^\infty t^{-\alpha-1} \d t \bigg).\label{eq:fractional_laplacian_bound_exp}
\end{eqnarray}
We handle the three integrals in~\eqref{eq:fractional_laplacian_bound_exp} one after the other now. First, consider
\begin{eqnarray*}
    & &\int_0^{\frac{2d(i,j)}{\rho(A)}} \frac{4\exp(-\frac{1}{4} \rho(A) t)}{\rho(A) t}\left(\frac{e \rho(A) t}{4d(i,j)}\right)^{d(i,j)}\cdot t^{-\alpha-1} \d t \nonumber\\
    &=& \frac{4}{\rho(A)} \left(\frac{e \rho(A)}{4d(i,j)}\right)^{d(i,j)} \int_0^{\frac{2d(i,j)}{\rho(A)}} \exp\left(-\frac14\rho(A)t\right)\cdot t^{d(i,j)-\alpha-2} \d t
\end{eqnarray*}
For a general function of the form $\exp(-mt)\cdot t^k$ with $m > 0, k > -1$, we find its antiderivative
\begin{equation*}
    \int \exp(-mt)\cdot t^k \d t = -m^{-k-1} \cdot \Gamma(k+1, mt) + c
\end{equation*}
for some constant $c$. Using the choice $m = \frac{\rho(A)}{4}$ and $k = d(i,j)-\alpha-2$  and inserting the limits of integration, this gives
\begin{eqnarray}
    & &\int_0^{\frac{2d(i,j)}{\rho(A)}} \exp\left(-\frac14\rho(A)t\right)\cdot t^{d(i,j)-\alpha-2} \d t \nonumber\\
    &=& -\left(\frac{4}{\rho(A)}\right)^{d(i,j)-\alpha-1} \left(\Gamma\left(d(i,j)-\alpha-1, \frac{d(i,j)}{2}\right) - \Gamma(d(i,j)-\alpha-1)\right)\nonumber\\
    &=& \left(\frac{4}{\rho(A)}\right)^{d(i,j)-\alpha-1} \gamma\left(d(i,j)-\alpha-1, \frac{d(i,j)}{2}\right).\label{eq:integral1_final}
\end{eqnarray}

The second integral in~\eqref{eq:fractional_laplacian_bound_exp} is of the general form $\exp(-\frac{m}{t})\cdot t^{-\alpha-1}$ with $m > 0$, for which we find the antiderivative
\begin{equation*}
    \int \exp(-mt)\cdot t^{-\alpha-1} \d t = m^{-\alpha} \cdot \Gamma\left(\alpha, \frac{m}{t}\right) + c
\end{equation*}
for some constant $c$. With the choice $m = \frac{4d(i,j)^2}{5\rho(A)}$ and inserting the limits of integration, we obtain
\begin{eqnarray}
& &\int_{\frac{2d(i,j)}{\rho(A)}}^{\frac{d(i,j)^2}{\rho(A)}} \exp\left(-\frac{4d(i,j)^2}{5\rho(A) t}\right)\cdot t^{-\alpha-1} \d t\nonumber\\
&=& \left(\frac{5\rho(A)}{4d(i,j)^2}\right)^\alpha \left(\Gamma\left(\alpha, \frac{4d(i,j)^2}{5\rho(A)} \cdot \frac{\rho(A)}{d(i,j)^2} \right) - \Gamma\left(\alpha, \frac{4d(i,j)^2}{5\rho(A)} \cdot \frac{\rho(A)}{2d(i,j)} \right)\right)\nonumber\\
&=& \left(\frac{5\rho(A)}{4d(i,j)^2}\right)^\alpha \left(\Gamma\left(\alpha, \frac{4}{5}\right) - \Gamma\left(\alpha, \frac{2d(i,j)}{5}\right)\right).\label{eq:integral2_final}
\end{eqnarray}

Finally, for the third integral, the antiderivative is simply
\begin{equation*}
\int t^{-\alpha-1} \d t = -\frac{1}{\alpha}t^{-\alpha} + c
\end{equation*}
for a constant $c$. After inserting the limits of integration, we directly obtain.
\begin{equation}
   \int_{\frac{d(i,j)^2}{\rho(A)}}^\infty t^{-\alpha-1} \d t = \frac{\rho(A)^\alpha}{\alpha \cdot d(i,j)^{2\alpha}} \label{eq:integral3_final}.
\end{equation}
Inserting~\eqref{eq:integral1_final},~\eqref{eq:integral2_final} and~\eqref{eq:integral3_final} into~\eqref{eq:fractional_laplacian_bound_exp} now yields~\eqref{eq:bounds_z_alpha}. 

The more compact formula~\eqref{eq:bounds_squareroot} for the special case $\alpha = \frac12$ directly follows by using  $\Gamma(\frac{1}{2}) = \sqrt{\pi}$ together with the relation $\Gamma(z,\frac{1}{2}) = \sqrt{\pi}\erfc(\sqrt{z})$.
\end{proof}

Numerical experiments illustrating the quality of the bounds obtained from Theorem~\ref{the:decay_z_alpha} will be given in Section~\ref{subsec:power_law}.

\subsection{The positive definite case}\label{subsec:positive_definite}
The bounds derived in Section~\ref{subsec:pos_semidefinite} are obviously also valid when $A$ is positive definite and not just semidefinite. However, in this case, the bounds can be sharpened by using the following observation. 

\begin{proposition}
Let $A \in \Cnn$ and let $\sigma \in \C$. Then
\begin{equation}\label{eq:exponential_shifted}
    \exp(A+\sigma I) = \exp(\sigma)\exp(A).
\end{equation}
\end{proposition}

Using~\eqref{eq:exponential_shifted}, we can first shift the smallest eigenvalue of $A$ to zero and then apply the result of Theorem~\ref{the:benzisimoncini_exp}. 

\begin{corollary}\label{cor:exp_posdef}
Let $A \in \Cnn$ be Hermitian positive definite  with smallest and largest eigenvalue $\lmin$ and $\lmax$, respectively. Denote by $d(i,j)$ the geodesic distance of the nodes $i$ and $j$ in the graph of $A$ and let $\rho := (\lmax-\lmin)/4$. Then for $i \neq j$ 
\begin{itemize}
\item[(i)] for $\rho t \geq 1$ and $\sqrt{4\rho t} \leq d(i,j) \leq 2\rho t$,
$$|[\exp(-t A)]_{ij}| \leq 10 \exp(-t\lmin)\exp\left(-\frac{1}{5\rho t}d(i,j)^2\right),$$
\item[(ii)] for $d(i,j) \geq 2\rho t$
$$|[\exp(-t A)]_{ij}| \leq 10 \frac{\exp(-(\rho+\lmin) t)}{\rho t}\left(\frac{e\rho t}{d(i,j)}\right)^{d(i,j)}.$$
\end{itemize}
\end{corollary}
\begin{proof}
Define the shifted matrix $\widetilde{A} = A-\lmin I$ with eigenvalues in $[0, \lmax-\lmin]$. From~\eqref{eq:exponential_shifted}, it then follows that 
$$\exp(-t\widetilde{A}) = \exp(-tA + t\lmin I) = \exp(t\lmin)\exp(-tA),$$
which is equivalent to $\exp(-tA) = \exp(-t\lmin)\exp(-t\widetilde{A})$. The result then follows by applying Theorem~\ref{the:benzisimoncini_exp} to $\exp(-t\widetilde{A})$.
\end{proof}

Corollary~\ref{cor:exp_posdef} directly gives rise to a result analogous to Lemma~\ref{lem:bernstein_decay_integral} for the positive definite case.

\begin{lemma}\label{lem:bernstein_decay_integral_posdef}
Let $f$ be a Bernstein function~\eqref{eq:bernstein_function}, let $A \in \Cnn$ be positive definite with smallest and largest eigenvalue $\lmin$ and $\lmax$, respectively. Denote by $d(i,j)$ the geodesic distance of the nodes $i$ and $j$ in the graph of $A$ and let $\rho := (\lmax-\lmin)/4$. Then for all $i,j$ with $d(i,j) \geq 2$, we have
\begin{eqnarray}
|[f(A)]_{ij}| &\leq& 10\int_0^{\frac{d(i,j)}{2\rho}} \frac{\exp(-(\rho+\lmin) t)}{\rho t}\left(\frac{e \rho t}{d(i,j)}\right)^{d(i,j)}\dmu \nonumber\\
& &+ 10\int_{\frac{d(i,j)}{2\rho}}^{\frac{d(i,j)^2}{4\rho}} \exp(-t\lmin)\cdot\exp\left(-\frac{d(i,j)^2}{5\rho t}\right) \dmu \nonumber\\
& & + \int_{\frac{d(i,j)^2}{4\rho}}^\infty \exp(-t\lmin)\dmu.\label{eq:bernstein_bound_exp_posdef}
\end{eqnarray}
\end{lemma}

In contrast to the integrals arising in the positive semidefinite case in Lemma~\ref{lem:bernstein_decay_integral}, even for the special case of fractional powers $z^\alpha$, there is no closed-form expression for the more complicated second integral in~\eqref{eq:bernstein_bound_exp_posdef}. Thus, in order to use Lemma~\ref{lem:bernstein_decay_integral_posdef} for predicting the decay in $f(A)$ for positive definite $A$, this integral needs to be evaluated by numerical quadrature.

\subsection{Explicit decay bounds for fractional powers of positive definite matrices via Cauchy--Stieltjes functions}\label{subsec:explicit_bounds_from_cauchy_stieltjes}

For fractional powers, explicit decay bounds can also be obtained in a different way when $A$ is positive definite, by employing a simple trick. For this, we write 
\begin{equation}\label{eq:fractional_powers_stieltjes}
    A^\alpha = A\cdot A^{\alpha-1}
\end{equation}
and exploit the fact that $A^{\alpha-1}$ is a Cauchy--Stieltjes function when $\alpha \in (0,1)$. Using the relation~\eqref{eq:fractional_powers_stieltjes}, known decay results for Cauchy--Stieltjes functions~\cite{BenziSimoncini2015,FrommerSchimmelSchweitzer2018} can easily be transferred to positive fractional powers. A similar trick is used in the context of extending the scope of restarted Krylov subspace methods for Stieltjes matrix functions in~\cite{FrommerGuettelSchweitzer2014a,FrommerGuettelSchweitzer2014b}.

\begin{theorem}\label{the:fractional_powers_via_cauchy_stieltjes}
 Let $A \in \Cnn$ be Hermitian positive definite with condition number $\kappa = \lmax/\lmin$, let $\alpha \in (0,1)$ and denote by $d(i,j)$ the geodesic distance of the nodes $i$ and $j$ in the graph of $A$. Then, for all $i, j$ with $d(i,j) \geq 2$,
 \begin{equation}\label{eq:fractional_powers_bound_stieltjes}
  |[A^{\alpha}]_{ij} | \leq 2 \lambda_{\min}^{\alpha-1} \|A\|_\infty \cdot q^{d(i,j)-1}\quad\mbox{ with } \quad q=\frac{\sqrt{\kappa}-1}{\sqrt{\kappa}+1}.
\end{equation}
\end{theorem}

\begin{proof}
Define $B = A^{\alpha-1}$, so that $A^{\alpha} = AB$, or, written element-wise,
\begin{equation}\label{eq:A_alpha_stieltjes_matrix_product}
    [A^\alpha]_{ij} = \sum\limits_{k = 1}^n a_{ik} b_{kj}.
\end{equation}

As $z^{\alpha-1}$ is a  Cauchy--Stieltjes function, we can apply~\cite[Theorem~4]{FrommerSchimmelSchweitzer2018} to $B$, which states that 
\begin{equation}\label{eq:bound_stieltjes}
    |b_{kj} | \leq 2 \lambda_{\min}^{\alpha-1} \cdot q^{d(k,j)},
\end{equation}
where $q$ is as defined in~\eqref{eq:fractional_powers_bound_stieltjes}. When $a_{ik} \neq 0$ we clearly have $d(k,j) \geq d(i,j)-1$. Using this relation after inserting~\eqref{eq:bound_stieltjes} into~\eqref{eq:A_alpha_stieltjes_matrix_product} and taking the absolute value gives
\begin{equation*}
    |[A^{\alpha}]_{ij}| \leq 2 \lambda_{\min}^{\alpha-1} \cdot q^{d(i,j)-1} \sum\limits_{k = 1}^n | a_{ik}| \leq 2 \lambda_{\min}^{\alpha-1} \cdot q^{d(i,j)-1} \|A\|_\infty,
\end{equation*}
which concludes the proof.
\end{proof}

Note that the technique used in the proof of Theorem~\ref{the:fractional_powers_via_cauchy_stieltjes} cannot be applied when $A$ is only positive semidefinite, as $A^{\alpha-1}$ is not defined when $A$ has a zero eigenvalue. Thus, the result cannot be extended to this situation.

\begin{remark}
The bound~\eqref{eq:A_alpha_stieltjes_matrix_product} is stated in a rather simple form that is valid for all $i,j$ with $d(i,j) \geq 2$. When one is interested in a specific entry $|[A^\alpha]_{ij}|$, the bound can be sharpened to
\begin{equation*}
  |[A^{\alpha}]_{ij} | \leq 2 \lambda_{\min}^{\alpha-1} \cdot \min\{\|A_{i:}\|_1, \|A_{:j}\|_1 \}  \cdot q^{d(i,j)-1},
\end{equation*}
with $q$ as in~\eqref{eq:fractional_powers_bound_stieltjes}, where $A_{i:}, A_{:j}$ denote the $i$th row and $j$th column of $A$, respectively. This directly follows from the fact that both the bounds obtained from writing $A^\alpha = AA^{1-\alpha}$ and from writing $A^\alpha = A^{1-\alpha}A$ are valid for each entry, so that one can always select the smaller of the two.
\end{remark}

\begin{example}\label{ex:laplace2d}
\begin{figure}
\centering
\includegraphics[width=.99\linewidth]{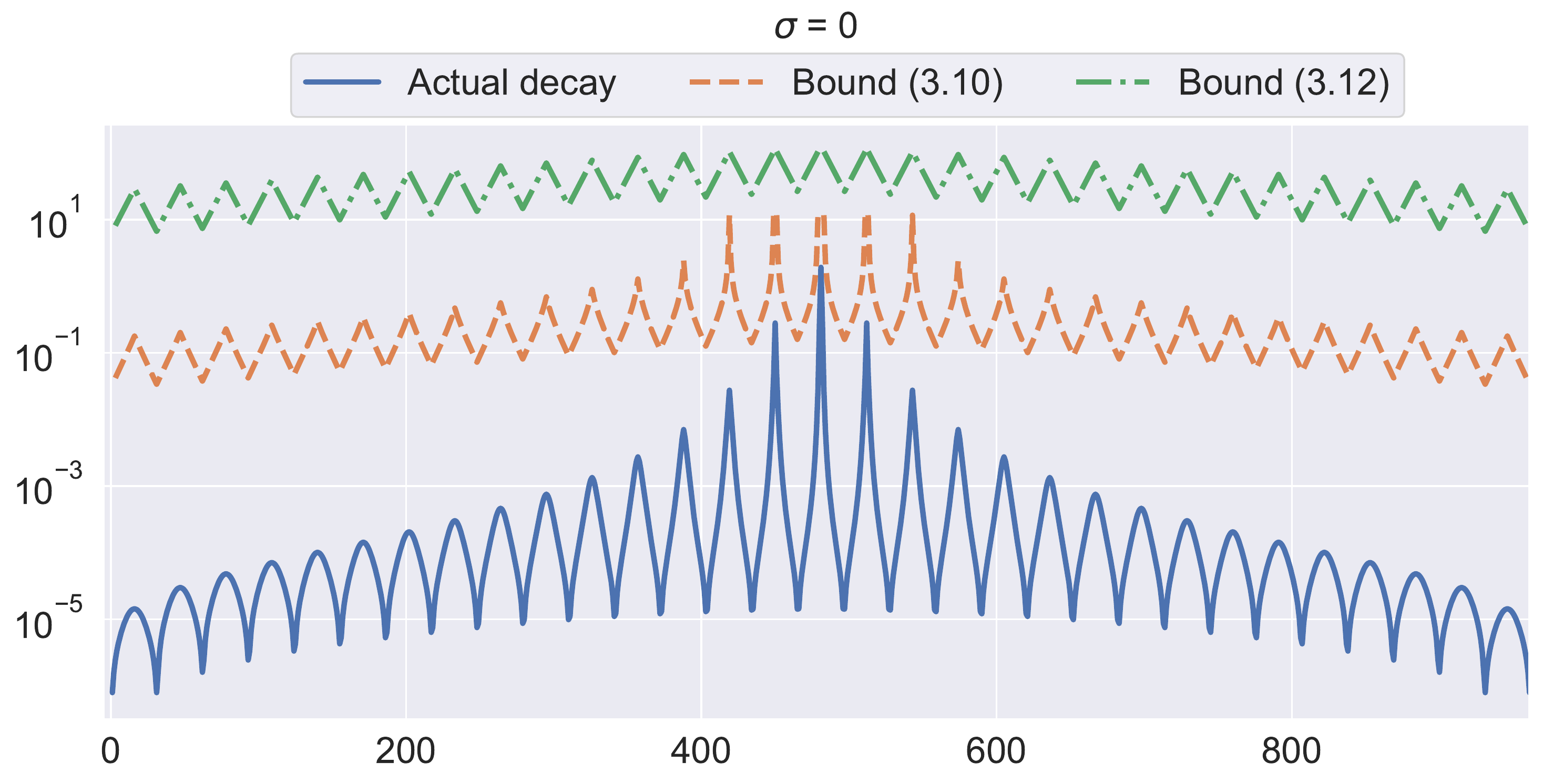}

\vspace{.5cm}
\includegraphics[width=.99\linewidth]{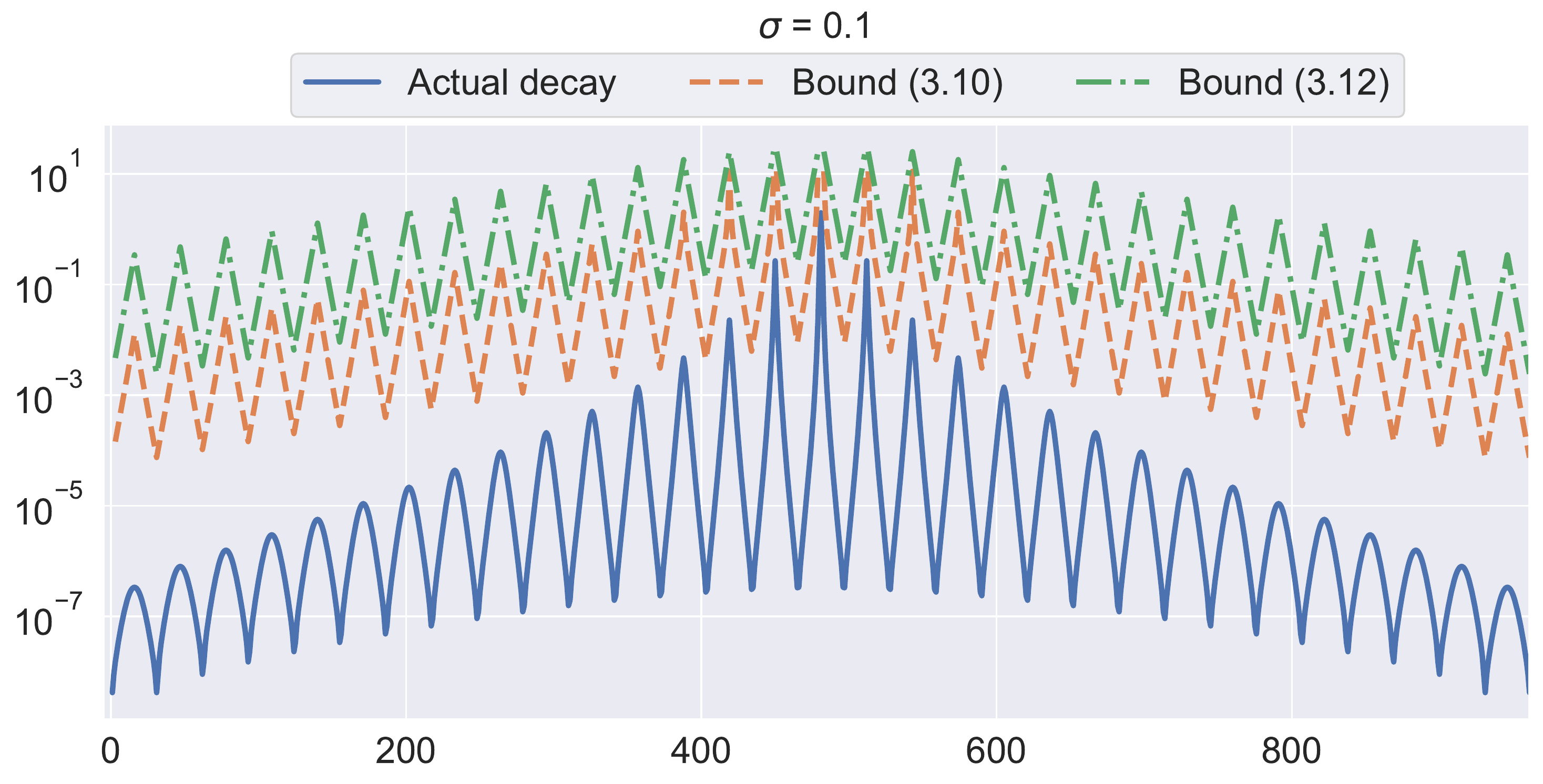}

\vspace{.5cm}
\includegraphics[width=.99\linewidth]{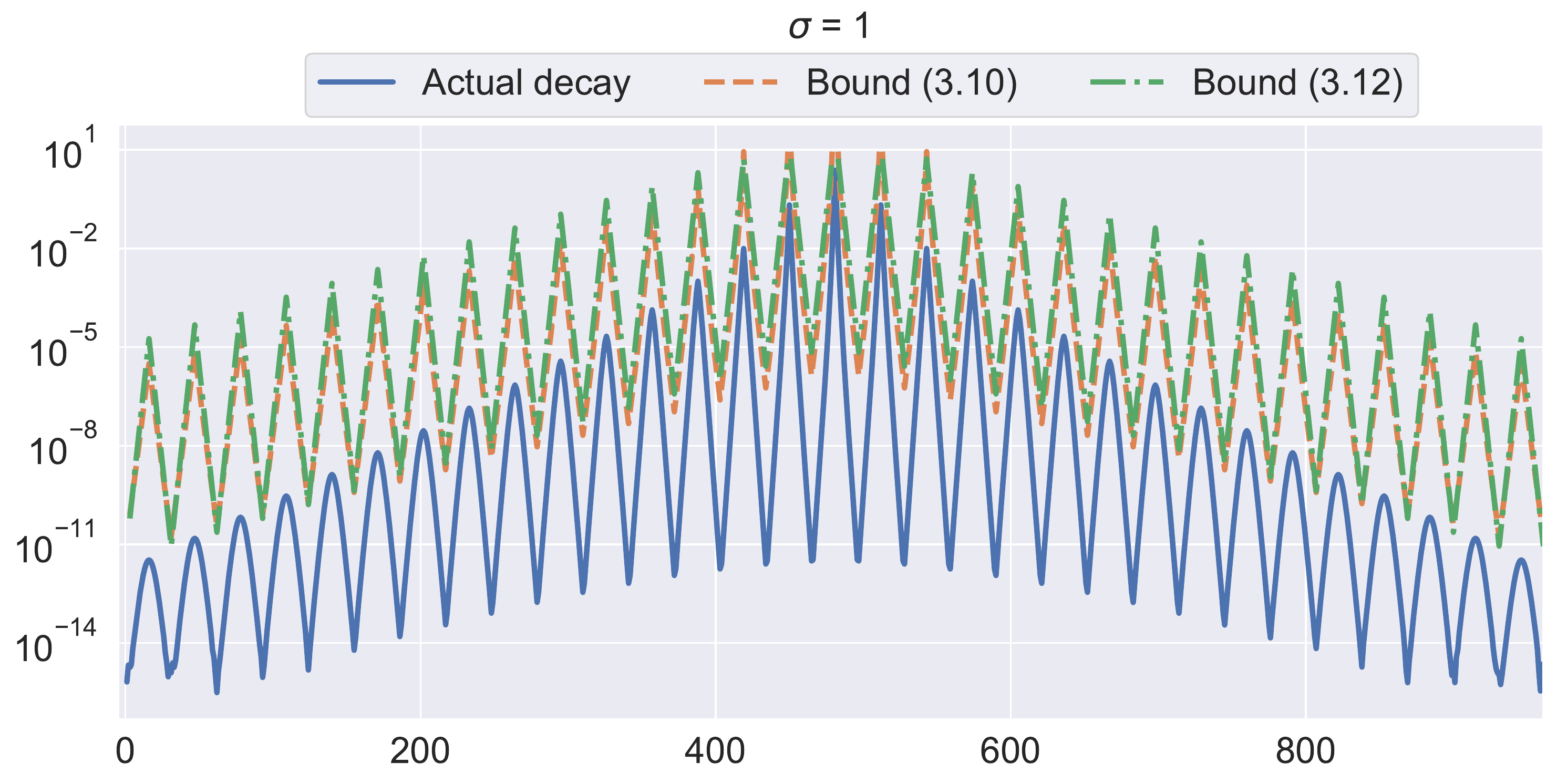}
\caption{
Decay in one column of $\sqrt{A}$, where $A = I \otimes M + M \otimes I \in C^{961 \times 961}$, where $M = \tridiag(-1,2 + \sigma,-1) \in \C^{31 \times 31}$ for $\sigma = 0$ (top), $\sigma = 0.1$ (center) and $\sigma = 1$ (bottom). The graph of $A$ is a regular two-dimensional grid of size $31\times 31$ and the depicted column corresponds to the node in the center of this grid.}
\label{fig:decay_2dlaplace}
\end{figure}
To illustrate our decay bounds for the positive definite case, we examine a simple model problem that is frequently used for demonstrating the quality of decay bounds~\cite{BenziSimoncini2015,FrommerSchimmelSchweitzer2021}. Let $A = I \otimes M + M \otimes I \in \C^{N^2 \times N^2}$, where $M = \tridiag(-1,2 + \sigma,-1) \in \C^{N \times N}$ and $\sigma \geq 0$. For the shift choice $\sigma = 0$, the matrix $A$ corresponds to the discretization of the two-dimensional Laplace operator on a regular square grid with homogeneous Dirichlet boundary conditions (up to a scaling). Increasing the shift $\sigma$ makes the resulting matrix $A$ better conditioned. We are interested in decay in the matrix square root $\sqrt{A}$, which plays an important role in Dirichlet-to-Neumann maps.

For our experiment, we choose $N = 31$, resulting in a matrix $A$ of size $961 \times 961$ and compare our bounds~\eqref{eq:bernstein_bound_exp_posdef} and~\eqref{eq:fractional_powers_bound_stieltjes} for the three parameters $\sigma = 0$, $\sigma = 0.1$ and $\sigma = 1$. We consider the magnitude of the entries of the column of $\sqrt{A}$ that belongs to the node in the center of the graph of $A$, i.e., at grid position $(16, 16)$. The integrals in the bound~\eqref{eq:bernstein_bound_exp_posdef} are approximated using the general-purpose quadrature routine \texttt{quad} from \texttt{SciPy.integrate}. The results of this experiment are depicted in Figure~\ref{fig:decay_2dlaplace}. Note that the seemingly ``oscillatory'' behavior of the entries of $\sqrt{A}$ is caused by the row-wise numbering of grid nodes. If plotted over the two-dimensional grid, one would observe a smooth decay with respect to the geodesic graph distance, as expected.

We observe that the integral bound~\eqref{eq:bernstein_bound_exp_posdef} always lies below the bound~\eqref{eq:fractional_powers_bound_stieltjes}, with the distance between the bounds reducing when the shift $\sigma$ is increased. Another interesting observation is that for $\sigma = 0$, the bound~\eqref{eq:bernstein_bound_exp_posdef} much better resembles the actual slope of the decay for nodes nearby the center of the grid, as it is not restricted to simple exponential decay of the form $C\cdot q^{d(i,j)}$.\hfill$\diamond$
\end{example}

\section{Application to the fractional graph Laplacian}\label{sec:fractional_laplacian}

Given an undirected graph $G = (V, E)$, the \emph{graph Laplacian} $L_G$ of $G$ is defined as
\begin{equation*}
L_G = D_G - A_G,
\end{equation*}
where $A_G$ is the adjacency matrix of $G$ and $D_G$ is a diagonal matrix containing the degrees of the nodes of $G$ on the diagonal. The graph Laplacian has applications in modeling diffusion processes on graphs, but also in spectral clustering~\cite{VonLuxburg2007}, graph drawing algorithms~\cite{Koren2003} and many other areas. As all row sums of the graph Laplacian are equal to zero, it is necessarily a singular matrix, and it is well known that it is always positive semidefinite~\cite{Merris1994}.

Recently, interest in the \emph{fractional graph Laplacian} has emerged, which allows to model nonlocal diffusion processes on graphs or use nonlocal random walks for the exploration of large networks~\cite{BenziBertacciniDurastanteSimunec2019,BenziSimunec2021,BianchiDonatelliDurastanteMazza2021,BertacciniDurastante2021,Estrada2021}. The fractional graph Laplacian is simply defined by taking a fractional power of the ordinary Laplacian, i.e., by $L_G^\alpha$. As $L_G^\alpha$ is a singular $M$-matrix~\cite{BenziBertacciniDurastanteSimunec2019} with all entries nonzero (if $G$ is connected), it can be interpreted as the Laplacian of a \emph{weighted, fully-connected graph} $G^\alpha$ on the same set $V$ of nodes. In this context, decay bounds for the entries of $L_G^\alpha$ are of interest because they give insight into the nature of the connection strength in $G^\alpha$ between nodes that were not connected in $G$. 

\subsection{Power law decay in the fractional Laplacian}\label{subsec:power_law}
In~\cite{BenziBertacciniDurastanteSimunec2019, BianchiDonatelliDurastanteMazza2021} it was observed that the entries of the fractional graph Laplacian $L_G^\alpha$ exhibit a power-law decay away from the sparsity pattern of $L_G$. In particular, e.g., the following result was shown, which is based on Jackson's theorem~\cite{Meinardus1967}.
\begin{theorem}[Corollary~3.1 in~\cite{BenziBertacciniDurastanteSimunec2019}]\label{the:decay_jackson}
Let $L_G$ be the Laplacian of an undirected graph $G$ and let $\alpha \in (0,1)$. Then, if $d(i,j) \geq 2$, we have
\begin{equation}\label{eq:decay_jackson}
|(L_G^{\alpha})_{ij}| \leq c \cdot \left(\frac{\rho(L_G)}{2}\right)^\alpha\cdot (d(i,j)-1)^{-\alpha}
\end{equation}
with $c = 1+\pi^2/2$.
\end{theorem}

We now compare our new result, Theorem~\ref{the:decay_z_alpha}, to Theorem~\ref{the:decay_jackson}. An important observation concerning~\eqref{eq:bounds_z_alpha} and~\eqref{eq:bounds_squareroot} is that the first term in the sum goes to zero exponentially in $d(i,j)$, so that asymptotically, the second and third term control the decay behavior in $L_G^{\alpha}$. So Theorem~\ref{the:decay_z_alpha} gives an asymptotic decay estimate of the form
\begin{equation}\label{eq:asymptotic_bound_square_root}
|[L_G^{\alpha}]_{ij}| \lesssim C \cdot d(i,j)^{-2\alpha},
\end{equation}
where $C$ is a constant, showing a power-law decay, as already observed in~\cite{BenziBertacciniDurastanteSimunec2019,BianchiDonatelliDurastanteMazza2021}, but with the improved exponent $-2\alpha$ instead of $-\alpha$. Thus, our new bounds show that the strength of connection between far apart nodes in $G^\alpha$ must actually drop off faster than known so far.

\begin{remark}
The bound~\eqref{eq:asymptotic_bound_square_root} only holds in an asymptotic sense, because we ignore the influence of the first term in~\eqref{eq:bounds_z_alpha}. It can, however, also easily be cast into an explicit, non-asymptotic form. Taking, e.g., $\alpha = \frac{1}{2}$, we have that
\begin{equation*}
    \frac{e^d}{d^d}\gamma\left(d-3/2, \frac{d}{2}\right) \leq d^{-1}
\end{equation*}
holds for all $d \geq 4$. Thus, by Theorem~\ref{the:decay_z_alpha} we directly find a rigorous, non-asymptotic bound of the form
\begin{equation*}
|[\sqrt{L_G}]_{ij}| \leq \widetilde{C} \cdot d(i,j)^{-1} \text{ for all $i,j$ with $d(i,j) \geq 4$}
\end{equation*}
with a modified constant $\widetilde{C}$.
\end{remark}

\begin{example}\label{ex:1dchain}
To illustrate how our new decay estimates compare to those from~\cite{BenziBertacciniDurastanteSimunec2019,BianchiDonatelliDurastanteMazza2021}, we begin by considering a very simple test problem: Let $G$ be a one-dimensional chain of length $n$. The corresponding graph Laplacian is the tridiagonal matrix
\begin{equation*}
L_G = \left[\begin{array}{ccccc}
1 & -1 & & &\\
-1 & 2 & \ddots & &\\
 & \ddots & \ddots & \ddots & \\
 & & -1 & 2 & -1 \\
& & & -1 & 1 
\end{array}\right] \in \Rnn.
\end{equation*}

\begin{figure}
\centering
\includegraphics[width=.99\linewidth]{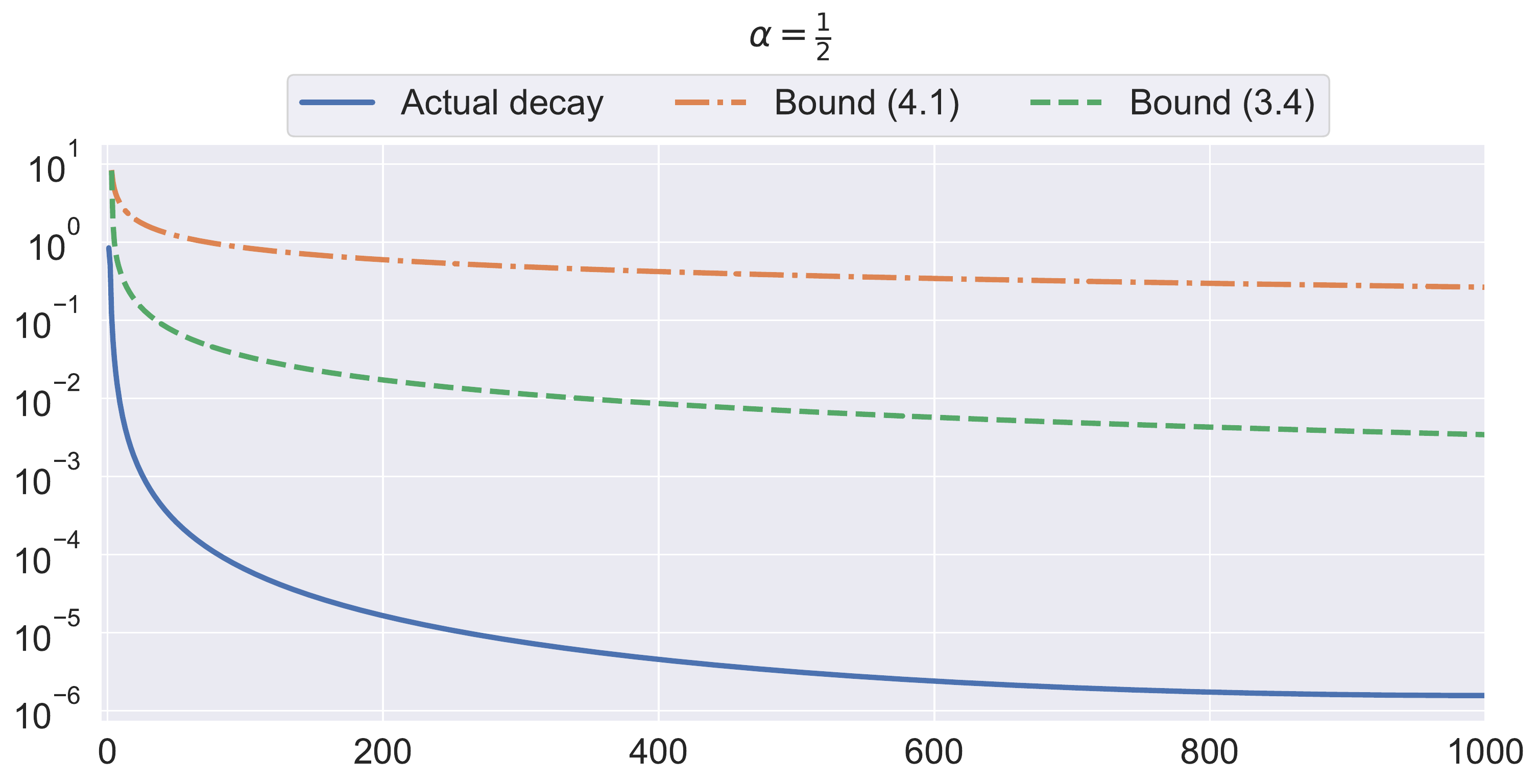}

\vspace{.5cm}
\includegraphics[width=.99\linewidth]{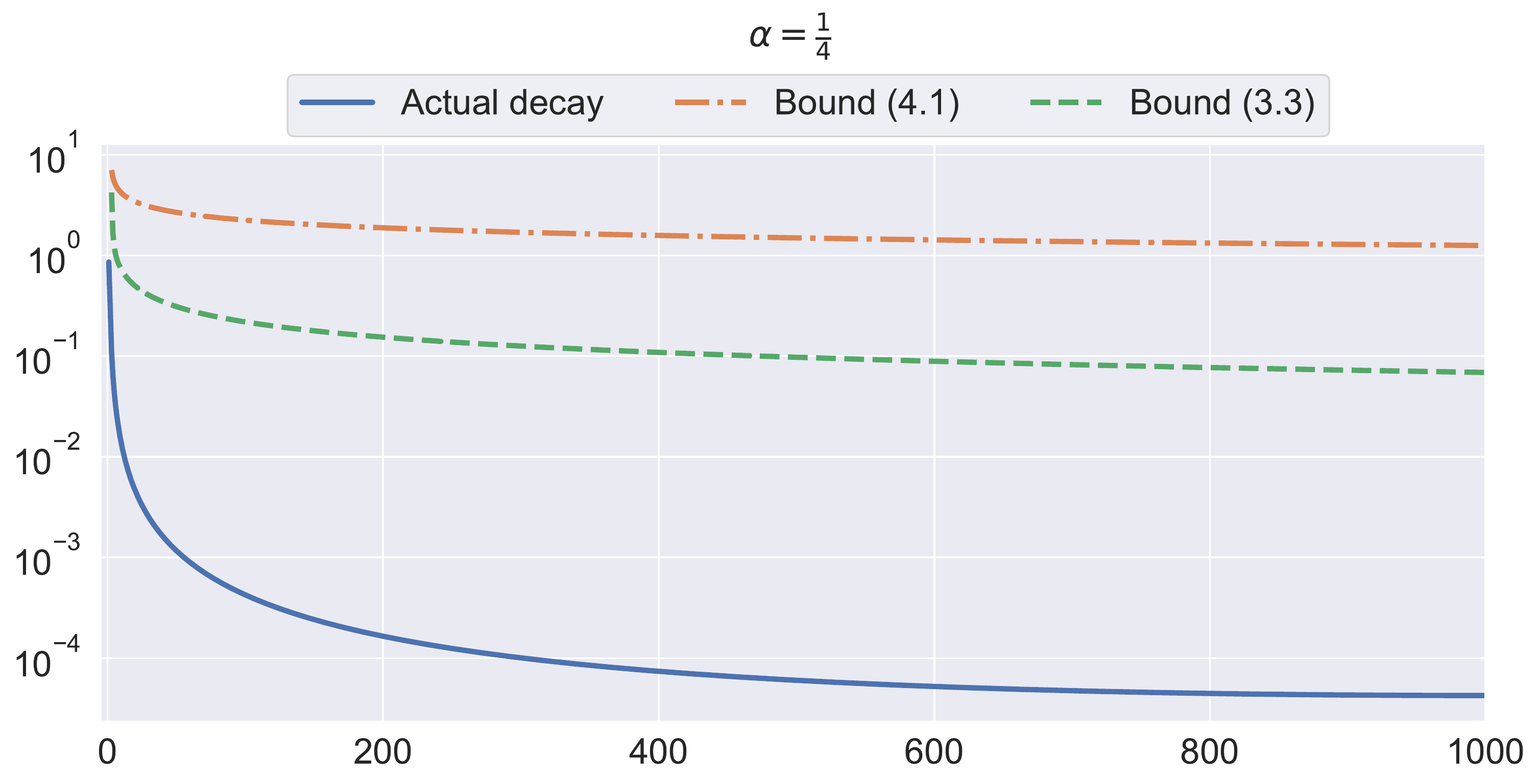}
\caption{
Decay in the first column of $L_G^\alpha$ where $L_G$ is the Laplacian corresponding to a one-dimensional chain and $\alpha = \frac12$ (top) or $\alpha = \frac14$ (bottom).}
\label{fig:decay_chain}
\end{figure}

By the Ger\v{s}gorin disk theorem, the spectral radius of $L_G$ is bounded by $\rho(L_G)\leq 4$ independently of $n$. In Figure~\ref{fig:decay_chain} we compare the bound of Theorem~\ref{the:decay_z_alpha} to that of Theorem~\ref{the:decay_jackson} for $\alpha = \frac12$ and $\alpha = \frac14$. In both cases we can observe that the slope of our new bound more closely resembles the actual decay behavior due to the additional factor 2 in the exponent. In addition, we also predict the order of magnitude of the entries better (though they are still overestimated by a quite large margin). Concerning this, it is also interesting to compare the constants involved in the two bounds. In Theorem~\ref{the:decay_jackson}, there is only one relevant constant, which is
$$\left(1+\frac{\pi^2}{2}\right)\sqrt{\frac{\rho(L_G)}{2}} \approx 8.39 \text{ for } \alpha = \frac12$$
and 
$$\left(1+\frac{\pi^2}{2}\right)\left(\frac{\rho(L_G)}{2}\right)^{\frac14} \approx 7.06  \text{ for } \alpha = \frac14.$$

In Theorem~\ref{the:decay_z_alpha}, several constants occur. As the first term does not play a role in the asymptotic behavior for growing $d$, we ignore it. The constant in front of the term $d(i,j)^{-1}$ can be estimated as 
$$\frac{5}{2}\sqrt{5\rho(L_G)}\erfc\left(\sqrt{\frac45}\right)+\sqrt{\frac{\rho(L_G)}{\pi}} \sim 3.43  \text{ for } \alpha = \frac12$$
because $\erfc\left(\sqrt{\frac{2d(i,j)}{5}}\right)$ goes to zero as $d(i,j)$ increases. For $\alpha = \frac14$, we obtain the constant
$$\frac{\alpha}{\Gamma(1-\alpha)}\left(10\left(\frac{5\rho(L_G)}{4}\right)^\alpha \cdot \Gamma\left(\alpha,\frac{4}{5}\right)+\frac{\rho(L_G)^\alpha}{\alpha}\right) \sim 2.18,$$
again making use of the fact that $\Gamma\left(\alpha,\frac{2d(i,j)}{5}\right)$ goes to zero for increasing $d(i,j)$.
Thus, the constants in the bounds have---at least asymptotically for large distances--- also improved by a factor of roughly $2.4$ and $3.2$, respectively.\hfill$\diamond$
\end{example}

\begin{example}\label{ex:geometric_graph}
\begin{figure}
\centering
\includegraphics[width=.30\linewidth]{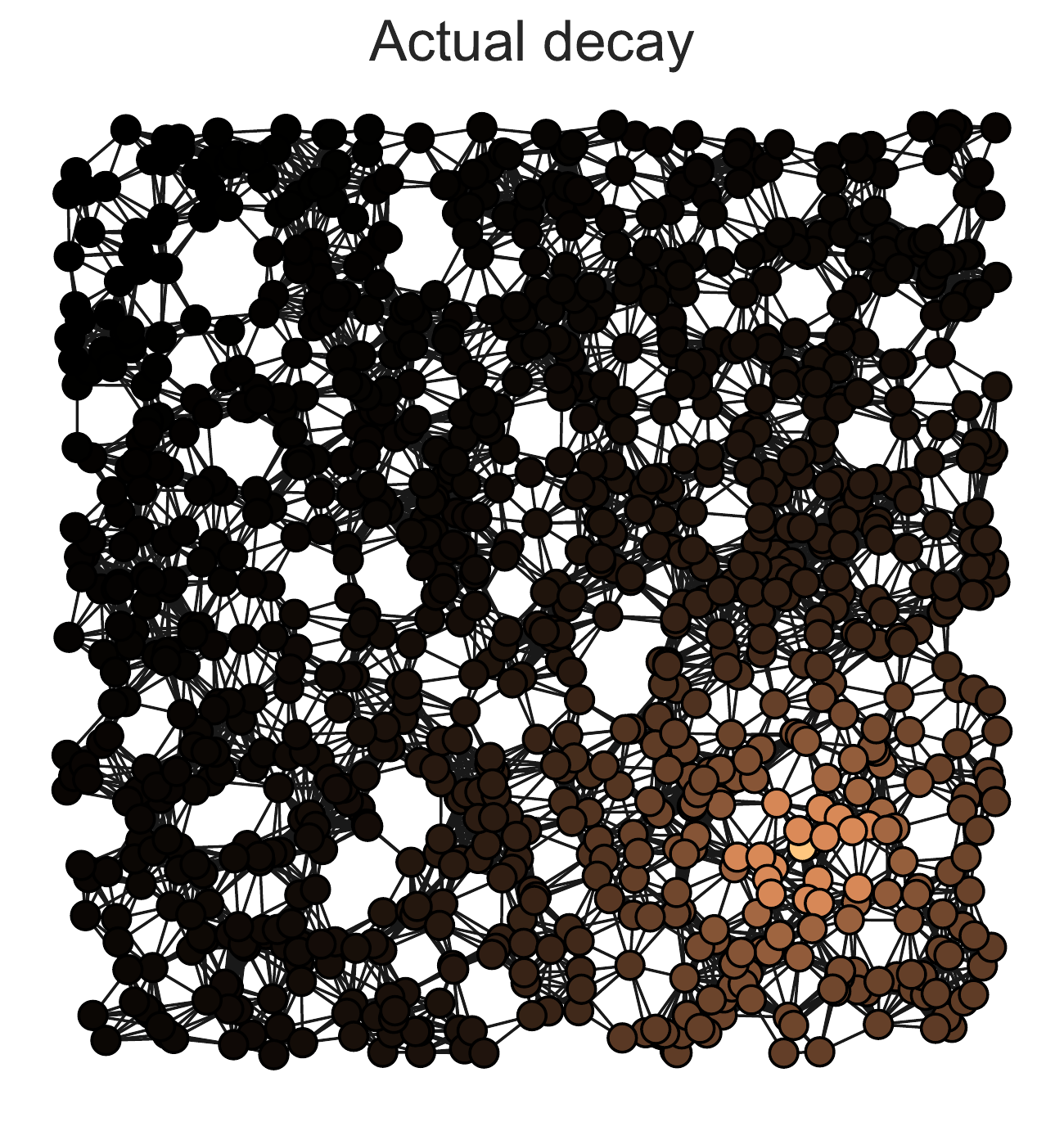}
\includegraphics[width=.30\linewidth]{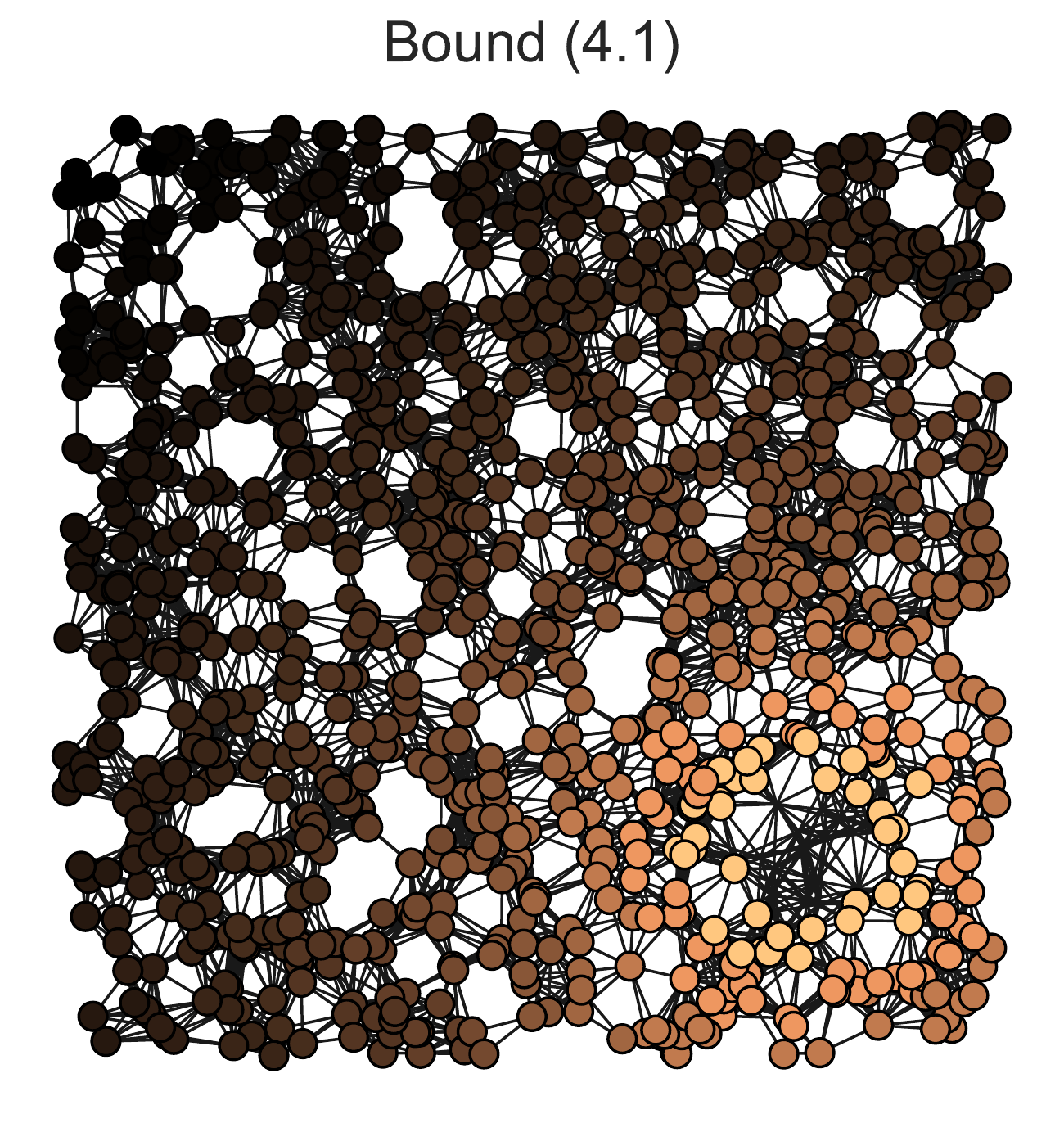}
\includegraphics[width=.35\linewidth]{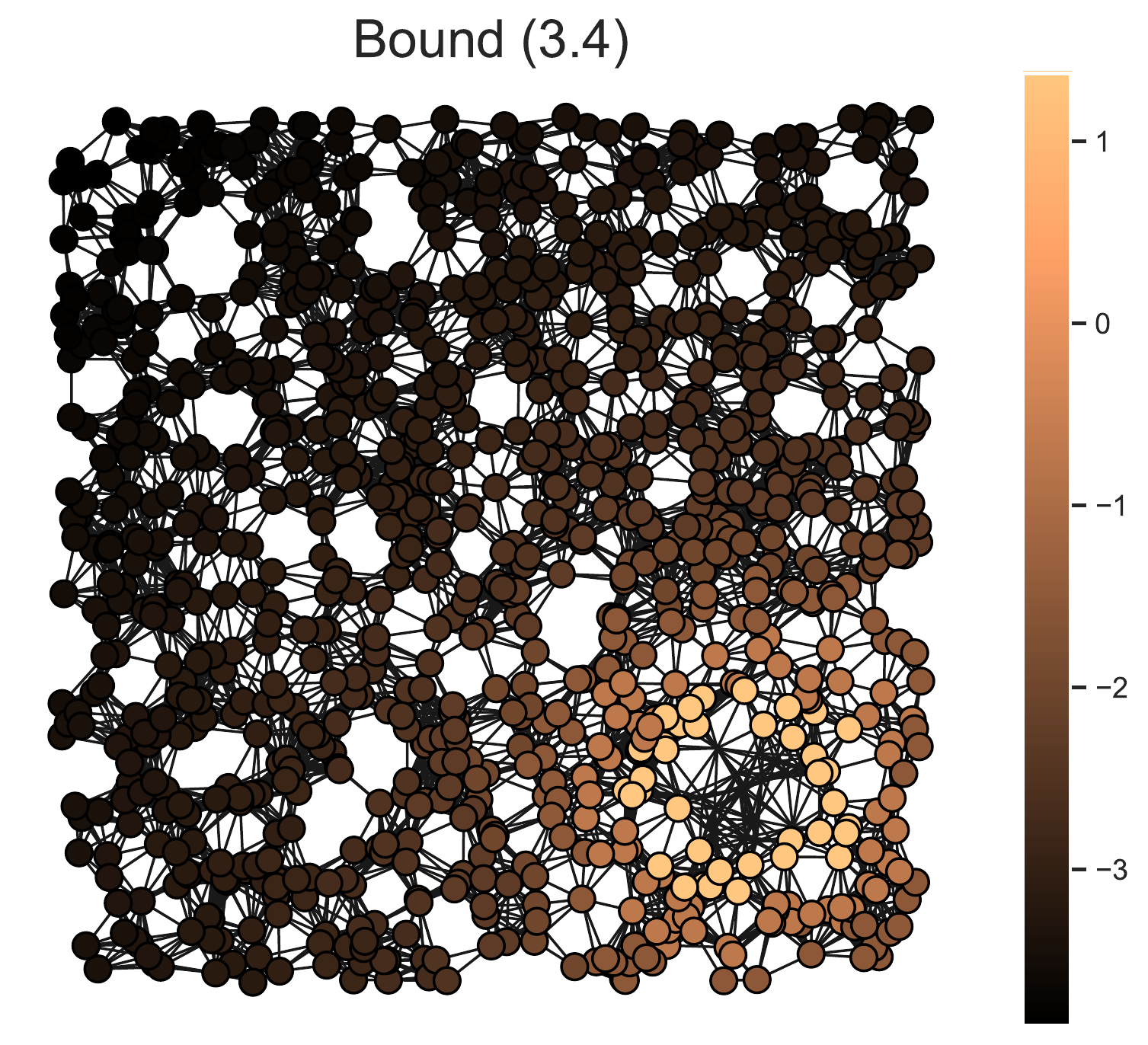}
\caption{Decay in one column of the fractional Laplacian $\sqrt{L_G}$ for a random geometric graph $G$ on a logarithmic scale. Left: Actual decay, Center: Bound~\eqref{eq:decay_jackson}, Right: Bound~\eqref{eq:bounds_squareroot}. In the center and right graph, nodes for which no bound is available are not drawn.}
\label{fig:decay_geometric_graph_draw}
\end{figure}
\begin{figure}
\centering
\includegraphics[width=.99\linewidth]{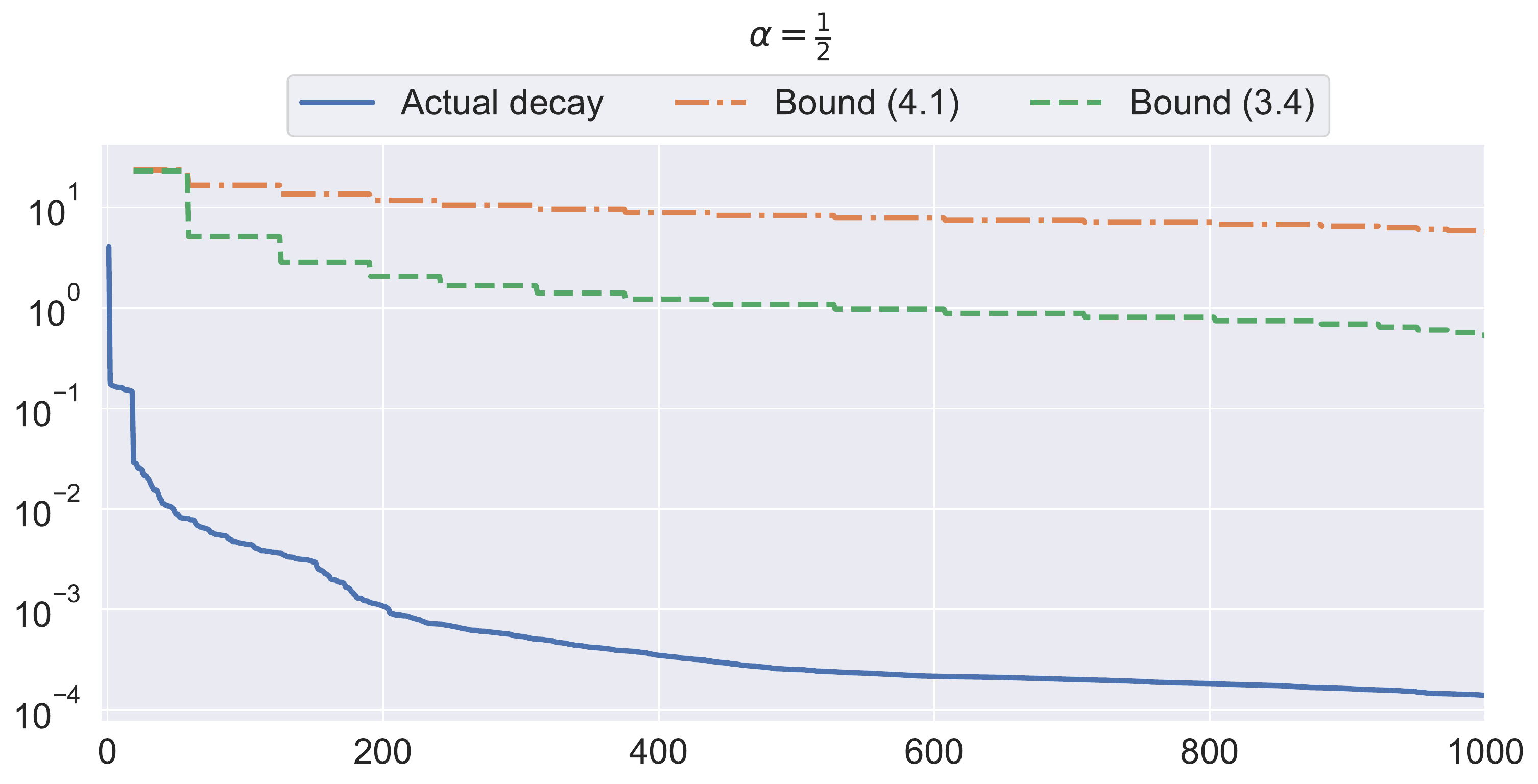}

\vspace{.5cm}
\includegraphics[width=.99\linewidth]{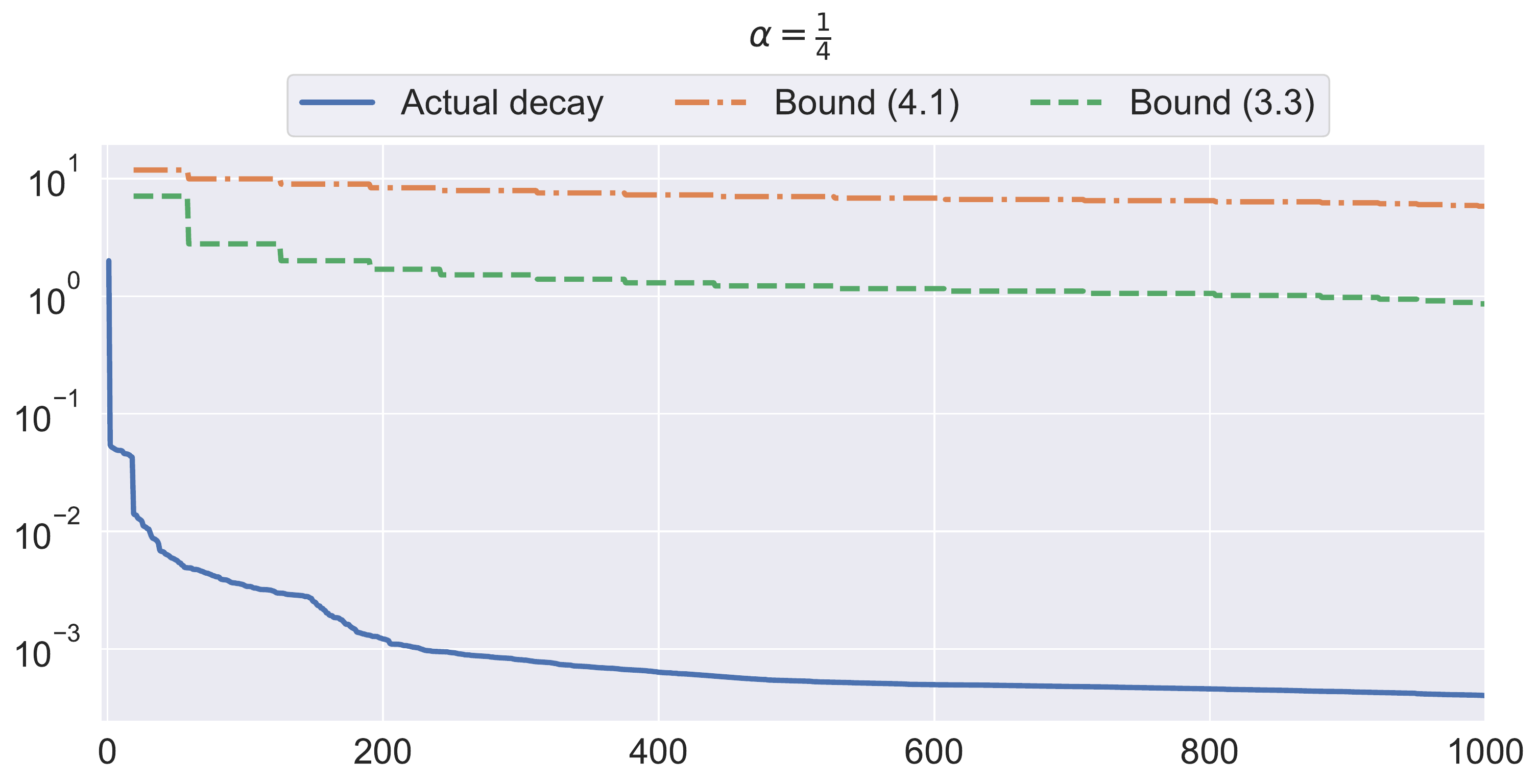}
\caption{
Decay in the first column of $L_G^\alpha$ where $L_G$ is the Laplacian corresponding to a random geometric graph $G$ (see text for details) and $\alpha = \frac12$ (top) or $\alpha = \frac14$ (bottom).}
\label{fig:decay_geometric_graph_lineplot}
\end{figure}
Next, we consider an example graph with a more irregular structure. We construct a \emph{random geometric graph} by sampling $n = 1000$ uniformly distributed points in the unit square and then connecting all pairs with distance below $0.075$ by an edge. The resulting graph is depicted in Figure~\ref{fig:decay_geometric_graph_draw}. The degrees of the nodes in this graph range from $4$ to $29$ and the spectral radius of its Laplacian is $\rho(L_G) \approx 31.64$. In the left-most plot in Figure~\ref{fig:decay_geometric_graph_draw}, the color coding of the nodes depicts the magnitude of the entries in one column of the fractional Laplacian $\sqrt{L_G}$ of this graph (on a logarithmic scale). The column corresponds to a node near the bottom right corner of the unit square (the node with brightest color in the plot). One can nicely see how the magnitude of the entries decays the farther one moves away from the source node, which is in line with what our decay bounds predict (and what one would intuitively expect). In the center part of Figure~\ref{fig:decay_geometric_graph_draw}, the decay bound~\eqref{eq:decay_jackson} is shown and in the right-most part, we plot our new bound~\eqref{eq:bounds_squareroot}. In both cases, nodes for which no bound is available (i.e., nodes with distance $1$ from the source node) are not drawn. Additionally, for easier comparison, we also show the magnitude of the entries and the bounds in a line plot in Figure~\ref{fig:decay_geometric_graph_lineplot}, which also contains results for the case $\alpha = \frac{1}{4}$. As the graph has no regular underlying structure, we order the nodes for this plot according to the magnitude of the corresponding column entries, so that the entries form a monotonically decreasing sequence, which makes it easier to make sense of the plot. Observe the resulting ``stair-case'' like structure of both bounds~\eqref{eq:bounds_z_alpha}/\eqref{eq:bounds_squareroot} and~\eqref{eq:decay_jackson}, which predict the same order of magnitude for all entries belonging to nodes that have the same distance from the source node. While less pronounced, we can also observe a similar structure with plateau-like areas in the actual decay.

Again our new bounds~\eqref{eq:bounds_z_alpha}/\eqref{eq:bounds_squareroot} more accurately predict the slope of the decay and also have a smaller magnitude than~\eqref{eq:decay_jackson}. However, both bounds overestimate the actual decay by quite a large margin and the actual decay slope is a little bit steeper than what the new bound predicts.\hfill$\diamond$
\end{example}

\subsection{Investigating the sharpness of the decay bounds}\label{subsec:sharpness}
An interesting question in the study of decay bounds is whether they are asymptotically optimal or whether there is a possibility for further improvement. Currently, we do not have a definitive answer to this question, but we give an illustrative example that suggests that a further improvement of the exponent in the power law might be possible.

\begin{example}\label{ex:cycle_explicit}
We consider a simple example graph for which we can derive analytical formulas for the entries of the fractional Laplacian. Let $G_n$ be a graph consisting of $n$ nodes arranged in a circle, where we assume that $n$ is odd. The Laplacian of this graph is given by
\begin{equation*}
L_{G_n} = \left[\begin{array}{ccccc}
2 & -1 & & &-1\\
-1 & 2 & \ddots & &\\
 & \ddots & \ddots & \ddots & \\
 & & -1 & 2 & -1 \\
-1 & & & -1 & 2 
\end{array}\right] \in \Rnn.
\end{equation*}
The eigenvalues and eigenvectors of this matrix are analytically known and given by
\begin{equation}\label{eq:eigenvalues_1dcycle}
\lambda_k = \begin{cases}
4\sin^2\left(\frac{\pi k}{2n}\right) & \text{if $k$ is even,}\\
4\sin^2\left(\frac{\pi(k-1)}{2n}\right) & \text{if $k$ is odd}
\end{cases}
\end{equation}
and
\begin{equation}\label{eq:eigenvectors_1dcycle}
\vv_{i,k} = \begin{cases}
n^{-1/2} & \text{if $k = 1$,}\\
\sqrt{\frac{2}{n}}\sin\left(\frac{\pi(i-\frac12)k}{n}\right) & \text{if $k$ is even,}\\
\sqrt{\frac{2}{n}}\cos\left(\frac{\pi(i-\frac12)(k-1)}{n}\right) & \text{if $k$ is odd,}
\end{cases}
\end{equation}
respectively. In particular, all eigenvalues except $\lambda_1 = 0$ appear twice and the spectrum of $L_{G_n}$ is contained in $[0, 4]$ independent of $n$. Using~\eqref{eq:eigenvalues_1dcycle} and~\eqref{eq:eigenvectors_1dcycle}, we can analytically compute entries of $\sqrt{L_{G_n}}$. We have
\begin{eqnarray}
[\sqrt{L_{G_n}}]_{ij} &=& \sum_{k = 1}^n \sqrt{\lambda_k} \vv_{i,k}\vv_{j,k} \nonumber\\
	       &=& \sum_{\ell = 1}^{\frac{n-1}{2}}  \sqrt{\lambda_{2\ell}} (\vv_{i,2\ell}\vv_{j,2\ell}+\vv_{i,2\ell+1}\vv_{j,2\ell+1})\nonumber\\
	       &=& \frac{4}{n}\sum_{\ell=1}^{\frac{n-1}{2}} \sin\left(\frac{\pi \ell}{n}\right) \bigg(\sin\left(\frac{2\pi(i-\frac12)\ell}{n}\right) \sin\left(\frac{2\pi(j-\frac12)\ell}{n}\right) \nonumber\\ 
	       & &\qquad\qquad \qquad\qquad\qquad+ \cos\left(\frac{2\pi(i-\frac12)\ell}{n}\right)\cos\left(\frac{2\pi(j-\frac12)\ell}{n}\right)\bigg)\nonumber\\
	       &=& \frac{4}{n} \sum_{\ell=1}^{\frac{n-1}{2}} \sin\left(\frac{\pi \ell}{n}\right) \cos\left(\frac{2\pi\ell(i-j)}{n}\right), \label{eq:sqrtL_ij}
\end{eqnarray}
where we used an angle sum identity for the last equality. We can resolve the summation in~\eqref{eq:sqrtL_ij} using several standard trigonometric identities, yielding
\begin{equation}\label{eq:sqrtL_ij2}
[\sqrt{L_{G_n}}]_{ij} = \frac{1}{n}\left(\cot\left(\frac{\pi(2(i-j)+1)}{2n}\right) + \cot\left(\frac{\pi(1-2(i-j))}{2n}\right)\right).
\end{equation}
For the $(\lceil\frac{n}{2}\rceil,1)$-entry, above formula~\eqref{eq:sqrtL_ij2} simplifies to 
\begin{equation}\label{eq:sqrtL_ij3}
[\sqrt{L_{G_n}}]_{\lceil\frac{n}{2}\rceil,1} = \frac{1}{n} \cot\left(\pi\left(\frac{1}{n}-\frac{1}{2}\right)\right).
\end{equation}
From L'H\^opital's rule, one can see that~\eqref{eq:sqrtL_ij3} implies that as $n$ goes to infinity, $[\sqrt{L_{G_n}}]_{\lceil\frac{n}{2}\rceil,1}$ goes to zero as $(\lceil\frac{n}{2}\rceil-1)^{-2}$. In contrast, the decay bound of Theorem~\ref{the:decay_z_alpha} predicts a decrease as $(\lceil\frac{n}{2}\rceil-1)^{-1}$. \hfill$\diamond$
\end{example}

In the numerical experiments in Example~\ref{ex:1dchain} and~\ref{ex:geometric_graph}---as well as in other numerical experiments not reported here---we observed that the actual decay was even faster than predicted by our new, refined bounds, and we were not able to find a graph for which the bound was (provably or experimentally) asymptotically sharp. Thus, motivated by Example~\ref{ex:cycle_explicit}, we conjecture that
\begin{equation*}
|[L_G^{\alpha}]_{ij}| \lesssim C \cdot d(i,j)^{-4\alpha}.
\end{equation*}
It remains an open topic for future research to prove or disprove this conjecture.

\section{Conclusions}\label{sec:conclusions}
We have derived new integral-based decay bounds for Bernstein functions of Hermitian matrices $A$, with special emphasis on the case that $A$ is positive semidefinite and $f(z) = z^\alpha, \alpha \in (0,1)$. In this case, analytic expressions for all invovled integrals are available, making the bounds particularly easy to use. In other cases, some of the integrals appearing in the bounds need to be evaluated by numerical quadrature. As a particularly important application, we have considered nonlocal network dynamics described by the fractional graph Laplacian. It is well-known that the strength of connection between far apart nodes in the network follows a power law in this case, and we were able to improve the exponent of this power law from $-\alpha $ to $-2\alpha$ using our new approach. Motivated by studying the closed form representation of the fractional Laplacian of a cycle graph, we conjectured that a further improvement up to an exponent of $-4\alpha$ could be possible.

\section*{Acknowledgement} 
The author wishes to thank Fabio Durastante and Andreas Frommer for helpful comments on an earlier version of the manuscript.

\bibliographystyle{siam}
\bibliography{matrixfunctions}

\newcommand{\noopsort}[1]{} \newcommand{\printfirst}[2]{#1}
  \newcommand{\singleletter}[1]{#1} \newcommand{\switchargs}[2]{#2#1}
\begin{thebibliography}{10}

\bibitem{BeckermannReichel2009}
{\sc B.~Beckermann and L.~Reichel}, {\em Error estimation and evaluation of
  matrix functions via the {F}aber transform}, SIAM J. Numer. Anal., 47 (2009),
  pp.~3849--3883.

\bibitem{Benzi2016}
{\sc M.~Benzi}, {\em Localization in {M}atrix {C}omputations: {T}heory and
  {A}pplications}, in Exploiting {H}idden {S}tructure in {M}atrix
  {C}omputations: {A}lgorithms and {A}pplications, M.~Benzi and V.~Simoncini,
  eds., vol.~2173 of C.I.M.E.\ Foundation Subseries, Springer, New York, 2016,
  pp.~211--317.

\bibitem{BenziBertacciniDurastanteSimunec2019}
{\sc M.~Benzi, D.~Bertaccini, F.~Durastante, and I.~Simunec}, {\em Non-local
  network dynamics via fractional graph {L}aplacians}, J.\ Complex Netw., 8
  (2020), p.~cnaa017.

\bibitem{BenziGolub1999}
{\sc M.~Benzi and G.~H. Golub}, {\em Bounds for the entries of matrix functions
  with applications to preconditioning}, BIT, 39 (1999), pp.~417--438.

\bibitem{BenziRazouk2007}
{\sc M.~Benzi and N.~Razouk}, {\em Decay bounds and {$O(n)$} algorithms for
  approximating functions of sparse matrices}, Electron.\ Trans.\ Numer.\
  Anal., 28 (2007), pp.~16--39.

\bibitem{BenziSimoncini2015}
{\sc M.~Benzi and V.~Simoncini}, {\em Decay bounds for functions of {H}ermitian
  matrices with banded or {K}ronecker structure}, SIAM J.\ Matrix Anal.\ Appl.,
  36 (2015), pp.~1263--1282.

\bibitem{BenziSimunec2021}
{\sc M.~Benzi and I.~Simunec}, {\em Rational {K}rylov methods for fractional
  diffusion problems on graphs}, BIT,  (2021), pp.~1--29.

\bibitem{Berg2007}
{\sc C.~Berg}, {\em {S}tieltjes-{P}ick-{B}ernstein-{S}choenberg and their
  connection to complete monotonicity}, in Positive Definite Functions. From
  Schoenberg to Space-Time Challenges, J.~Mateu and E.~Porcu, eds., Dept.\ of
  Mathematics, University Jaume I, Castell\'{o}n de la Plana, Spain, 2008.

\bibitem{BertacciniDurastante2021}
{\sc D.~Bertaccini and F.~Durastante}, {\em Nonlocal diffusion of variable
  order on graphs}, arXiv preprint arXiv:2110.05424,  (2021).

\bibitem{BianchiDonatelliDurastanteMazza2021}
{\sc D.~Bianchi, M.~Donatelli, F.~Durastante, and M.~Mazza}, {\em
  Compatibility, embedding and regularization of non-local random walks on
  graphs}, arXiv preprint arXiv:2101.00425,  (2021).

\bibitem{BowlerMiyazaki2012}
{\sc D.~R. Bowler and T.~Miyazaki}, {\em {$O(N)$} methods in electronic
  structure calculations}, Rep.\ Prog.\ Phys., 75 (2012), p.~036503.

\bibitem{DemkoMossSmith1984}
{\sc S.~Demko, W.~F. Moss, and W.~Smith}, {\em Decay rates for inverses of
  banded matrices}, Math.\ Comput., 43 (1984), pp.~491--499.

\bibitem{EijkhoutPolman1988}
{\sc V.~Eijkhout and B.~Polman}, {\em Decay rates of inverses of banded
  {$M$}-matrices that are near to {T}oeplitz matrices}, Linear Algebra Appl.,
  109 (1988), pp.~247--277.

\bibitem{Estrada2021}
{\sc E.~Estrada}, {\em Path {L}aplacians versus fractional {L}aplacians as
  nonlocal operators on networks}, New J.\ Phys., 23 (2021), p.~073049.

\bibitem{FordSavostyanovZamarashkin2014}
{\sc N.~J. Ford, D.~V. Savostyanov, and N.~L. Zamarashkin}, {\em On the decay
  of the elements of inverse triangular {T}oeplitz matrices}, SIAM J.\ Matrix
  Anal.\ Appl., 35 (2014), pp.~1288--1302.

\bibitem{FrommerGuettelSchweitzer2014b}
{\sc A.~Frommer, S.~G{\"u}ttel, and M.~Schweitzer}, {\em Convergence of
  restarted {K}rylov subspace methods for {S}tieltjes functions of matrices},
  SIAM J.\ Matrix Anal.\ Appl., 35 (2014), pp.~1602--1624.

\bibitem{FrommerGuettelSchweitzer2014a}
\leavevmode\vrule height 2pt depth -1.6pt width 23pt, {\em Efficient and stable
  {A}rnoldi restarts for matrix functions based on quadrature}, SIAM J.\ Matrix
  Anal.\ Appl., 35 (2014), pp.~661--683.

\bibitem{FrommerSchimmelSchweitzer2018}
{\sc A.~Frommer, C.~Schimmel, and M.~Schweitzer}, {\em Bounds for the decay of
  the entries in inverses and {C}auchy--{S}tieltjes functions of certain
  sparse, normal matrices}, Numer. Linear Algebra Appl., 25 (2018), p.~e2131.

\bibitem{FrommerSchimmelSchweitzer2018b}
\leavevmode\vrule height 2pt depth -1.6pt width 23pt, {\em Non-{T}oeplitz decay
  bounds for inverses of {H}ermitian positive definite tridiagonal matrices},
  Electron.\ Trans.\ Numer.\ Anal., 48 (2018), pp.~362--372.

\bibitem{FrommerSchimmelSchweitzer2021}
\leavevmode\vrule height 2pt depth -1.6pt width 23pt, {\em Analysis of probing
  techniques for sparse approximation and trace estimation of decaying matrix
  functions}, SIAM J.\ Matrix Anal.\ Appl., 42 (2021), pp.~1290--1318.

\bibitem{GiscardLuiThwaiteJaksch2015}
{\sc P.-L. Giscard, K.~Lui, S.~Thwaite, and D.~Jaksch}, {\em An exact
  formulation of the time-ordered exponential using path-sums}, J.\ Math.\
  Phys., 56 (2015), p.~053503.

\bibitem{GuettelKnizhnerman2013}
{\sc S.~G\"{u}ttel and L.~Knizhnerman}, {\em A black-box rational {A}rnoldi
  variant for {C}auchy--{S}tieltjes matrix functions}, BIT, 53 (2013),
  pp.~595--616.

\bibitem{GuettelSchweitzer2021}
{\sc S.~G\"uttel and M.~Schweitzer}, {\em A comparison of limited-memory
  {K}rylov methods for {S}tieltjes functions of {H}ermitian matrices}, SIAM J.\
  Matrix Anal.\ Appl., 42 (2021), pp.~83--107.

\bibitem{Higham2008}
{\sc N.~J. Higham}, {\em Functions of Matrices: Theory and Computation}, SIAM,
  Philadelphia, 2008.

\bibitem{HochbruckLubich1997}
{\sc M.~Hochbruck and {\relax Ch}.~Lubich}, {\em On {K}rylov subspace
  approximations to the matrix exponential operator}, SIAM J.\ Numer.\ Anal.,
  34 (1997), pp.~1911--1925.

\bibitem{Iserles2000}
{\sc A.~Iserles}, {\em How large is the exponential of a banded matrix?}, N.\
  Z.\ J.\ Math., 29 (2000), pp.~177--192.

\bibitem{Koren2003}
{\sc Y.~Koren}, {\em On spectral graph drawing}, in International Computing and
  Combinatorics Conference, Springer, 2003, pp.~496--508.

\bibitem{LopezPugliese2005}
{\sc L.~Lopez and A.~Pugliese}, {\em Decay behaviour of functions of
  skew-symmetric matrices}, Proceedings of HERCMA,  (2005), pp.~22--24.

\bibitem{MasseiRobol2020}
{\sc S.~Massei and L.~Robol}, {\em Rational {K}rylov for {S}tieltjes matrix
  functions: convergence and pole selection}, BIT,  (2020), pp.~1--37.

\bibitem{Meinardus1967}
{\sc G.~Meinardus}, {\em Approximation of Functions: Theory and Numerical
  Methods}, Springer, Berlin, 1967.

\bibitem{Merris1994}
{\sc R.~Merris}, {\em Laplacian matrices of graphs: {A} survey}, Linear Algebra
  Appl., 197 (1994), pp.~143--176.

\bibitem{PozzaSimoncini2019}
{\sc S.~Pozza and V.~Simoncini}, {\em Inexact {A}rnoldi residual estimates and
  decay properties for functions of non-{H}ermitian matrices}, BIT, 59 (2019),
  pp.~969--986.

\bibitem{RiascosMateos2014}
{\sc A.~P. Riascos and J.~L. Mateos}, {\em Fractional dynamics on networks:
  Emergence of anomalous diffusion and {L}{\'e}vy flights}, Physical Review E,
  90 (2014), p.~032809.

\bibitem{RiascosMateos2019}
\leavevmode\vrule height 2pt depth -1.6pt width 23pt, {\em Random walks on
  weighted networks: {E}xploring local and non-local navigation strategies},
  arXiv preprint arXiv:1901.05609,  (2019).

\bibitem{SchillingSongVondracek2012}
{\sc R.~L. Schilling, R.~Song, and Z.~Vondracek}, {\em Bernstein Functions --
  Theory and Applications}, De Gruyter, Berlin, Boston, 2012.

\bibitem{Schimmel2020}
{\sc C.~Schimmel}, {\em Bounds for the decay in matrix functions and its
  exploitation in matrix computations}, PhD thesis, Bergische Universit{\"a}t
  Wuppertal, Fakult{\"a}t f{\"u}r Mathematik und Naturwissenschaften, 2020.

\bibitem{VonLuxburg2007}
{\sc U.~Von~Luxburg}, {\em A tutorial on spectral clustering}, Statist.\
  Comput., 17 (2007), pp.~395--416.

\end{thebibliography}

\end{document}